\documentclass[11pt]{article}
\usepackage{amscd}
\usepackage{amsmath}
\usepackage{latexsym}
\usepackage{amsfonts}
\usepackage{amssymb}
\usepackage{amsthm}
\usepackage{graphicx}

\newtheorem{theorem}{Theorem}[section]
\newtheorem{lemma}{Lemma}[section]

\newtheorem{proposition}{Proposition}[section]
\newtheorem{definition}{Definition}[section]

\newtheorem{corollary}{Corollary}[section]
\newtheorem{remark}{Remark}[section]

\makeatletter
\newcommand{\Extend}[5]{\ext@arrow0099{\arrowfill@#1#2#3}{#4}{#5}}
\renewcommand{\labelenumi}{\Roman{enumi}.}
\makeatother

\begin{document}

 \title{ Global well-posedness and scattering for the mass-critical Hartree equation with radial data}
 \author{{Changxing Miao$^{\dag}$\ \ Guixiang Xu$^{\dag}$ \ \ and \ Lifeng Zhao $^{\ddag}$}\\
         {\small $^{\dag}$Institute of Applied Physics and Computational Mathematics}\\
         {\small P. O. Box 8009,\ Beijing,\ China,\ 100088}\\
         {\small $^\ddag$ Department of Mathematics, University of Science and Technology of China}\\
         {\small (miao\_changxing@iapcm.ac.cn, \ xu\_guixiang@iapcm.ac.cn, zhao\_lifeng@iapcm.ac.cn ) }\\
         \date{}
        }
\maketitle

 \noindent{\bf Abstract:}  We establish global
well-posedness and scattering for solutions to the mass-critical
nonlinear Hartree equation $iu_t+\Delta u=\pm(|x|^{-2}*|u|^2)u$
for large spherically symmetric $L^2_x(\Bbb{R}^d)$ initial data;
in the focusing case we require, of course, that the mass is
strictly less than that of the ground state. \vskip0.5cm

\noindent{\bf  R\'{e}sum\'{e}:} \;\; Nous \'etablissons
l'existence globale et la diffusion des solutions de l'\'equation
non lin\'eaire de masse critique de Hartree $iu_t+\Delta
u=\pm(|x|^{-2}*|u|^2)u$ pour des donn\'ees initiales grandes \`a
sym\'etrie sph\'erique dans $L^2_x(\Bbb{R}^d)$ ; dans le cas
focalisant nous imposons, bien s\^ur, que la masse soit
strictement inf\'erieure \`a celle de l'\'etat fondamental.



 \section{Introduction}
 We primarily consider the mass-critical Hartree equation:
 \begin{equation} \label{har}
\left\{ \aligned
    iu_t +  \Delta u  & = \mu (|x|^{-2}*|u|^2)u, \quad  \text{in}\  \mathbb{R}^d \times \mathbb{R},\\
     u(0)&=u_0(x), \quad \text{in} \ \mathbb{R}^d.
\endaligned
\right.
\end{equation}
where $d\geq3$, $\mu=\pm1$, with $\mu=+1$ known as the defocusing
case and $\mu=-1$ as the focusing case.  The Hartree equation arises
in the study of Boson stars and other physical phenomena, see for
example \cite{pi}. In chemistry, it appears as a continous-limit
model for mesoscopic molecular structures, see \cite{grc}.
\begin{definition} A function
$u:I\times\Bbb{R}^d\rightarrow\mathbb{C}$ on a non-empty time
interval $I\subset\Bbb{R}$ (possibly infinite or semi-infinite) is a
strong $L^2(\Bbb{R}^d)$ solution to (\ref{har}) if it lies in the
class $C_t^0L_x^2(K\times\Bbb{R}^d)\cap
L_t^6L_x^\frac{6d}{3d-2}(K\times\Bbb{R}^d)$ for all compact
$K\subset I$, and we have the Duhamel formula
\begin{equation}u(t_1)=e^{i(t_1-t_0)\Delta}u(t_0)-i\int_{t_0}^{t_1}e^{i(t_1-t)\Delta}F(u(t))dt
\end{equation}
for all $t_0$, $t_1\in I$, where $F(u)=\mu (|x|^{-2}*|u|^2)u$. We
refer to the interval $I$ as the {\it lifespan} of $u$. We say that
$u$ is a {\it maximal-lifespan solution} if the solution cannot be
extended to any strictly larger interval. We say that u is a {\it
global solution} if $I=\Bbb{R}$.
\end{definition}
\begin{definition}
We say that a solution $u$ to (\ref{har}) blows up forward in time
if there exists a time $t_0\in I$ such that
\begin{equation}\|u\|_{L_t^6L_x^\frac{6d}{3d-2}([t_0,\sup(I))\times{\Bbb R}^d)}=\infty
\end{equation}
and that $u$ blows up backward in time if there exists a time
$t_0\in I$ such that
\begin{equation}\|u\|_{L_t^6L_x^\frac{6d}{3d-2}((\inf(I),t_0]\times{\Bbb
R}^d)}=\infty.
\end{equation}
\end{definition}

For the mass-critical Hartree equation (\ref{har}) with data in
$L^2$, the authors obtained some well-posedness and scattering
results in \cite{MiXZ1} using the method of \cite{CW}. We collect
these facts as follows:
\begin{theorem}[Local well-posedness]Given $u_0\in L_x^2({\Bbb
R}^d)$ and $t_0\in \Bbb R$, there exists a unique maximal-lifespan
solution $u$ to (\ref{har}) with $u(t_0)=u_0$. We will write $I$ for
the maximal lifespan. This solution also has the following
properties:\\
$\bigstar$ (Local existence) $I$ is an open neighborhood of $t_0$.\\
$\bigstar$ (Mass conservation) The solution $u$ has a conserved
mass: for $t\in I$,
\begin{equation}\label{mass}M(u)=M(u(t)):=\int_{{\Bbb R}^d}|u(t,x)|^2dx.\end{equation}
$\bigstar$ (Blowup criterion) If $\sup(I)$ is finite, then $u$ blows
up forward in time; if $\inf(I)$ is finite, then $u$ blows up
backward in time.\\
$\bigstar$ (Continuous dependence) If $u_0^{(n)}$ is a sequence
converging to $u_0$ in $L_x^2({\Bbb R}^d)$ and $u^{(n)}:
I^{(n)}\times {\Bbb R}^d\rightarrow\Bbb C$ are the associated
maximal-lifespan
solutions, then $u^{(n)}$ converges locally uniformly to $u$.\\
$\bigstar$ (Scattering) If $u$ does not blow up forward in time,
then $\sup(I)=+\infty$ and $u$ scatters forward in time, that is,
there exists a unique $u_+\in L_x^2({\Bbb R}^d)$ such that
$$\lim_{t\rightarrow+\infty}\|u(t)-e^{it\Delta}u_{+}\|_{L_x^2({\Bbb R}^d)}=0.$$ Similarly, if $u$ does not blow up
backward in time, then $\inf(I)=-\infty$ and $u$ scatters backward
in time, that is, there is a unique $u_-\in L_x^2({\Bbb R}^d)$ so
that
$$\lim_{t\rightarrow-\infty}\|u(t)-e^{it\Delta}u_{-}\|_{L_x^2({\Bbb R}^d)}=0.$$ $\bigstar$ (Spherical symmetry) If $u_0$ is spherically
symmetric, then $u$ remains spherically symmetric for all time.\\
$\bigstar$ (Small data global existence) If $M(u_0)$ is sufficiently
small, then $u$ is a global solution which does not blow up either
forward or backward in time. Indeed, in this case$$\int_{\Bbb
R}\Big(\int_{{\Bbb
R}^d}|u(t,x)|^\frac{6d}{3d-2}dx\Big)^\frac{3d-2}{d}dt\lesssim
M(u)^3.$$
\end{theorem}

From the small data global existence, we conclude that for the
mass-critical Hartree equation (\ref{har}), there exists a minimal
mass $m_0$ such that solutions with mass strictly smaller than $m_0$
are global and scatter in time. It is conjectured that $m_0$ should
be $+\infty$ in the defocusing case and be $M(Q)$ in the focusing
case, where $Q$ is the {\it ground state}, that is, the positive
radial Schwartz solution $Q$ to the elliptic equation
\begin{equation}\label{ground}\Delta Q+(|x|^{-2}*|Q|^2)Q=Q.
\end{equation}
%
%
%

In this paper we prove the conjecture for radial data. In
particular, we have
\begin{theorem}\label{main}In the defocusing case $\mu=+1$, all
maximal-lifespan radial solutions to (\ref{har}) are global and do
not blow up either forward or backward in time. In the focusing case
$\mu=-1$, all maximal-lifespan radial solutions to (\ref{har}) with
$M(u)<M(Q)$ are global and do not blow up either forward or backward
in time.
\end{theorem}
In fact, the result in Theorem \ref{main} is sharp. $e^{it}Q$ is the
solution to (\ref{har}) that blows up at infinity. Moreover, this
equation is invariant under the pseudo-conformal transformation
\begin{equation*}
u(t,x)\longmapsto
(i(t-T))^{-\frac{d}{2}}e^{\frac{i|x|^2}{4(t-T)}}\bar{u}\big(\frac{1}{t-T},
\frac{x}{t-T}\big).
\end{equation*}
So
$(i(t-T))^{-\frac{d}{2}}e^{\frac{i|x|^2}{4(t-T)}}e^{-\frac{i}{t-T}}Q(\frac{x}{t-T})$
is the solution that blows up at finite time $t=T$ for fixed $T$.

In the proof of the above theorem, we adapt the ideas and techniques
in \cite{KTV} and \cite{KVZ}, which represent the state of the art
in nonlinear dispersive equations. In \cite{KTV}, R. Killip, T. Tao
and M. Visan established the global well-posedness and scattering
for radial solutions to mass-critical nonlinear Schr\"odinger
equations in dimension $d=2$. R. Killip, M. Visan and X. Zhang
extended this result to higher dimensions in \cite{KVZ}. In
addition, C. E. Kenig, F. Merle dealt with the focusing
energy-critical nonlinear Schr\"odinger equation with radial data in
\cite{KeM06s}. For other related works, see S. Keraani \cite{Ker2},
T. Tao, M. Visan and X. Zhang \cite{TVZ1}, \cite{TVZ2}. Before we
state our argument, we need some definitions.
\begin{definition}[Symmetry group] For any phase
$\theta\in\mathbb{R}/2\pi\mathbb{Z}$, position $x_0\in\mathbb{R}^d$,
frequency $\xi_0\in\mathbb{R}^d$ and scaling parameter $\lambda>0$,
we define the unitary transformation $g_{\theta, \xi_0, x_0,
\lambda}: L_x^2(\mathbb{R}^d)\rightarrow L_x^2(\mathbb{R}^d)$ by the
formula
$$[g_{\theta, \xi_0, x_0,\lambda}f](x):=\frac{1}{\lambda^{d/2}}e^{i\theta}e^{ix\cdot \xi_0}f\big(\frac{x-x_0}{\lambda}\big).$$
We let $G$ be the collection of such transformations. We also let
$G_{rad}\subset G$ denote the collection of transformations in $G$
which preserve spherical symmetry, or more explicitly,
$$G_{rad}:=\{g_{\theta,0,0,\lambda}: \theta\in\mathbb{R}/2\pi\mathbb{Z}; \lambda>0\}.$$
\end{definition}
\begin{definition}[Almost periodicity modulo symmetries]\label{aps} A solution
$u$ with lifespan $I$ is said to be {\it almost periodic modulo G}
if there exist (possibly discontinous) functions $N:
I\rightarrow\Bbb R^+$, $\xi: I\rightarrow{\Bbb R}^d$, $x:
I\rightarrow{\Bbb R}^d$ and a function $C: \Bbb R^+\rightarrow\Bbb
R^+$ such that
$$\int_{|x-x(t)|\geq C(\eta)/N(t)}|u(t,x)|^2dx\leq\eta$$
and
$$\int_{|\xi-\xi(t)|\geq C(\eta)N(t)}|\hat{u}(t,\xi)|^2d\xi\leq\eta$$
for all $t\in I$ and $\eta>0$. We refer to the function $N$ as the
{\it frequency scale function} for the solution $u$, $\xi$ as the
{\it frequency center function}, and $C$ as the {\it compactness
modulus function}. Furthermore, if we can select $x(t)=\xi(t)=0$ for
all $t\in I$, then we say that $u$ is {\it almost periodic modulo
$G_{rad}$}.
\end{definition}

\begin{remark}
By Ascoli-Arzela theorem, the above definition is equivalent to
either of the following two statements:
\begin{enumerate}
\item The quotient orbit $\Big\{Gu(t): t\in I\Big\}$ is a
precompact set of $G\backslash L^2$, where $G\backslash L^2$ is the
moduli space of $G$-orbits $Gf:=\{gf: g\in G\}$ of $L^2(\Bbb R^d)$.
\item There exists a compact subset $K$ of $L^2$
such that $u(t)\in GK$ for all $t\in I$; equivalently there exists a
group function $g:I\rightarrow G$ and a compact subset $K$ such that
$g^{-1}(t)u(t)\in K$ for any $t\in I$.
\end{enumerate}
\end{remark}
Suppose for contradiction that Theorem \ref{main} is not true, then
we can find an almost periodic solution. The solution must be one of
the following three forms:
\begin{theorem}[Three special scenarios for blowup]\label{3sc} Suppose Theorem
\ref{main} failed for spherically symmetric solutions, then there
exists a maximal-lifespan solution $u$ which may be chosen to be
spherically symmetric and almost periodic modulo $G_{rad}$.
Moreover, it blows up both forward and backward in time, and in the
focusing case also obeys $M(u)<M(Q)$.

With spherical symmetry, we can also ensure that the lifespan $I$
and the frequency scale function $N: I\rightarrow\mathbb{R}^+$ match
one of the following three scenarios: \begin{enumerate}\item
(Soliton-like solution) We have $I=\mathbb{R}$ and
$$N(t)=1$$ for
all $t\in \mathbb{R}$.\\
\item (Double high-to-low frequency cascade) We have $I=\mathbb{R}$,
$$\liminf_{t\rightarrow-\infty}N(t)=\liminf_{t\rightarrow+\infty}N(t)=0,$$
and $$\sup_{t\in\mathbb{R}}N(t)<\infty.$$
\item (Self-similar solution) We have $I=(0, +\infty)$ and
$$N(t)=t^{-1/2}$$ for all $t\in I$.
\end{enumerate}
\end{theorem}

This is a wonderful classification theorem first given in \cite{KTV}
although some other authors have mentioned some of them, see
\cite{Bou}, \cite{CKSTT}, \cite{KeM06s}, \cite{Vis}, \cite{TVZ2},
etc. In view of this theorem, our goal is to preclude the
possibilities of all the scenarios.

 Note that the minimal mass blow-up solution has very good
properties because it is localized in both physical and frequency
space. In fact, it admits higher regularity.
 \begin{theorem}[Regularity in the self-similar case]\label{self-similar} Let $u$ be a
 spherically symmetric solution to (\ref{har}) that is almost
 periodic modulo $G_{rad}$ and self-similar in the sense of
 Theorem \ref{3sc}. Then $u(t)\in H^s(\mathbb{R}^d)$ for all $t\in(0,
 \infty)$ and all $s\geq0$.
 \end{theorem}
 \begin{theorem}[Regularity in the global case]\label{tm51}Let $u$ be a global
spherically symmetric solution to (\ref{har}) that is almost
periodic modulo $G_{rad}$. Suppose also that $N(t)\lesssim 1$ for
all $t\in\Bbb R$, then $u\in L_t^\infty H^s(\Bbb R\times{\Bbb R}^d)$
for all $s\geq0$.
\end{theorem}
In the proofs of these two theorems for mass-critical Schr\"odinger
equations in \cite{KTV}, the radial assumption is fully exploited
based on a careful observation that there is a dichotomy between
scattering solutions and almost periodic solutions. There are
similar results for the mass-critical Hartree equation. More
precisely, one has:
\begin{proposition}\label{p2} Let $u: I\times\Bbb R^d\rightarrow\Bbb
C$ be a maximal-lifespan solution which is almost periodic modulo
$G$. Then $e^{-it\Delta}u(t)$ is weakly convergent to zero in
$L_x^2(\Bbb R^d)$ as $t\rightarrow\sup(I)$ or $t\rightarrow\inf(I)$.
\end{proposition}
As a corollary of Proposition \ref{p2}, we have
\begin{corollary}[A Duhamel formula]\label{co1} Let $u$ be a solution
 to (\ref{har}) which is almost periodic
 modulo $G$. Its maximal-lifespan is $I$. Then for all
$t\in I$, \begin{align}\label{duh} u(t)&=\lim_{T\nearrow \sup
I}i\int_t^Te^{i(t-t')\Delta}F(u(t'))dt'\\
&=-\lim_{T\searrow \inf I}i\int_t^Te^{i(t-t')\Delta}F(u(t'))dt'
\end{align}
as weak limit in $L_x^2$.
\end{corollary}
There are some new difficulties in dealing with the mass-critical
Hartree equation. One of them comes from the asymptotic
orthogonality. In the study of the mass-critical Hartree equation,
we have to use the non-symmetric spacetime norm because the
symmetric spacetime norm will lead to the restriction on dimension.
However, the orthogonality can be destroyed by the non-symmetric
spacetime norm.

We illustrate this by considering the simple example: Let
$\varphi_1$ and $\varphi_2$ are two bump functions in $\Bbb
R\times\Bbb R^2$. Let $x_n^1, x_n^2\in\Bbb R^2$ be such that
$|x_n^1-x_n^2|\rightarrow\infty$ as $n\rightarrow\infty$. Then we
have
\begin{equation}\label{11}\|\varphi_1(t,x+x_n^1)+\varphi_2(t,x+x_n^2)\|_{L_{t,x}^4}^4\rightarrow
\|\varphi_1\|_{L_{t,x}^4}^4+\|\varphi_2\|_{L_{t,x}^4}^4,
\end{equation} as $n\rightarrow\infty$. However, if we replace
$L_{t,x}^4$ with $L_t^6L_x^3$, then
\begin{align*}
&\|\varphi_1(t,x+x_n^1)+\varphi_2(t,x+x_n^2)\|_{L_t^6L_x^3}^6\\
=&\int\Big(\int|\varphi_1(t,x+x_n^1)+\varphi_2(t,x+x_n^2)|^3
dx\Big)^2dt\\
=&\int\Big(\int\big(|\varphi_1(t,x+x_n^1)|^3+|\varphi_2(t,x+x_n^2)|^3\\
&\quad+
3|\varphi_1(t,x+x_n^1)|^2|\varphi_2(t,x+x_n^2)|+3|\varphi_1(t,x+x_n^1)||\varphi_2(t,x+x_n^2)|^2\big)
dx\Big)^2dt\\
\rightarrow&\int\Big(\int\big(|\varphi_1(t,x+x_n^1)|^3+|\varphi_2(t,x+x_n^2)|^3\big)
dx\Big)^2dt\\
=&\int\Big(\int|\varphi_1(t,x)|^3dx\Big)^2dt+\int\Big(\int|\varphi_2(t,x)|^3dx\Big)^2dt\\
&\quad\quad+2\int\Big(\int|\varphi_1(t,x)|^3dx\Big)\Big(\int|\varphi_2(t,x)|^3dx\Big)dt\\
\nrightarrow&
\|\varphi_1\|_{L_t^6L_x^3}^6+\|\varphi_2\|_{L_t^6L_x^3}^6, \quad
\text{as}\ \ n\rightarrow\infty.
\end{align*}
Fortunately, for radial solution of the free Schr\"odinger equation,
there are only two kinds of orthogonality - time translation and
scaling ({\bf NO} spatial translation!), both of which are possessed
by time variable. So the orthogonality can be exploited and get the
desired orthogonal relation similar to (\ref{11}) (see Section 3).
So the radial assumption is necessary to prove Theorem \ref{3sc},
which is in contrast to \cite{KTV}, where the similar theorem was
established without the radial assumption. Such assumption is also
used in precluding the three enemies in the sense of Theorem
\ref{3sc}.

Some other difficulties coming from the convolution in the
nonlinearity lie on the fact that it's non-local in physical space
and singular in frequency space. For example, in precluding the
self-similar solution, we need to deal with terms such as
$\big(V*|u_{lo}^2|\big)u_{hi}$, where $\hat{u}_{lo}$ is supported in
$\{\xi: |\xi|\leq M\}$ and $\hat{u}_{hi}$ is supported in $\{\xi:
|\xi|\geq N\}$. In \cite{KTV}, the corresponding term
$|u_{hi}|^2u_{lo}$ can be estimated by means of bilinear estimate
(Lemma \ref{bi}). However, the convolution prevents the direct
interaction between $u_{lo}$ and $u_{hi}$ in Hartree equation, so
the bilinear estimate cannot be applied. In fact, to overcome the
difficulty we exploit Shao's estimate (Lemma \ref{sha}) and its dual
estimate in full strength to replace the bilinear estimates (see
Section 5). Meanwhile, we adapt weighted Strichartz estimate
(\ref{st}) in obtaining the additional regularity for the double
high-to-low frequency cascade and soliton-like solutions, where the
non-locality of the nonlinearity forces us to apply such estimates
in different regions (see Section 6).

The rest of the paper is organized as follows: In Section 2, we
record some known results such as basic facts in harmonic analysis,
various versions of Strichartz estimates and in/out decomposition.
In Section 3, we give the stability theory and the concentration
compactness result. In Section 4, we show that any failure of
Theorem \ref{main} must be ``caused'' by almost periodic solutions.
In Section 5, we preclude the self-similar solution by proving that
it possesses additional regularity. In Section 6, we prove the
additional regularity in the other two cases. In Section 7 and
Section 8, we preclude the double high-to-low frequency and
soliton-like solutions.

\section{Preliminaries}

\subsection{Some Notations}
We use $X\lesssim Y$ or $Y\gtrsim X$ whenever $X\leq CY$ for some
constant $C>0$. We use $O(Y)$ to denote any quantity $X$ such that
$|X|\lesssim Y$. We use the notation $X\sim Y$ whenever $X\lesssim
Y\lesssim X$. If $C$ depends upon some additional parameters, we
will indicate this with subscripts; for example, $X\lesssim_u Y$
denotes the assertion that $X\leq C_uY$ for some $C_u$ depending on
$u$. We use the `Japanese bracket' convention $\langle
x\rangle=(1+|x|^2)^{1/2}$. We write $L_t^qL_x^r$ to denote the
Banach space with norm
$$\|u\|_{L_t^qL_x^r(\mathbb{R}\times\mathbb{R}^3)}:=\Big(
\int_{\mathbb{R}}\big(\int_{\mathbb{R}^3}|u(t,x)|^rdx\big)^{q/r}dt\Big)^{1/q},$$
with the usual modifications when $q$ or $r$ are equal to infinity,
or when the domain $\mathbb{R}\times\mathbb{R}^d$ is replaced by
spacetime slab such as $I\times\mathbb{R}^d$. When $q=r$ we
abbreviate $L_t^qL_x^q$ as $L_{t,x}^q$.
\subsection{Basic harmonic analysis}
We recall some basic facts in Littlewood-Paley theory. Let
$\varphi(\xi)$ be a radial bump function supported in the ball
$\{\xi\in\mathbb{R}^d: |\xi|\leq\frac{11}{10}\}$ and equal to 1 on
the ball $\{\xi\in\mathbb{R}^d: |\xi|\leq1\}$. For each number
$N>0$, we define the Fourier multipliers
\begin{align*}
\widehat{P_{\leq N}f}(\xi)&:=\varphi(\xi/N)\hat{f}(\xi),\\
\widehat{P_{\geq N}f}(\xi)&:=(1-\varphi(\xi/N))\hat{f}(\xi),\\
\widehat{P_{N}f}(\xi)&:=(\varphi(\xi/N)-\varphi(2\xi/N))\hat{f}(\xi)
\end{align*}
and similarly $P_{<N}$ and $P_{\geq N}$. We also define
$$P_{M<\cdot\leq N}:=P_{\leq N}-P_{\leq M}=\sum_{M<N'<N}P_{N'}$$
whenever $M<N$. We will usually use these multipliers when $M$ and
$N$ are dyadic numbers; in particular, all summations over $N$ or
$M$ are understood to be over dyadic numbers. Nevertheless, it will
occasionally be convenient to allow $M$ and $N$ to not be a power of
2. Note that $P_N$ is not truly a projection; to get around this, we
will occasionally need to use fattened Littlewood-Paley operators:
\begin{equation}
\tilde{P}_N:=P_{N/2}+P_N+P_{2N}.
\end{equation}
They obey $P_N\tilde{P}_N=\tilde{P}_NP_N=P_N$.

As all Fourier multipliers, the Littlewood-Paley operators commute
with the propagator $e^{it\Delta}$, as well as with differential
operators such as $i\partial_t+\Delta$. We will use basic properties
of these operators many times, including
\begin{lemma}[Bernstein estimates]\label{Bern} For $1\leq p\leq q\leq\infty$,
\begin{align*}\||\nabla|^{\pm s}P_Nf\|_{L_x^p(\mathbb{R}^d)}&\sim
N^{\pm s}\|P_Nf\|_{L_x^p(\mathbb{R}^d)},\\
\|P_{\leq N}f\|_{L_x^q(\mathbb{R}^d)}&\lesssim
N^{\frac{d}{p}-\frac{d}{q}}\|P_{\leq N}f\|_{L_x^p(\mathbb{R}^d)},\\
\|P_{N}f\|_{L_x^q(\mathbb{R}^d)}&\lesssim
N^{\frac{d}{p}-\frac{d}{q}}\|P_{N}f\|_{L_x^p(\mathbb{R}^d)}.
\end{align*}
\end{lemma}

\subsection{Strichartz estimates}
Naturally, everything that we do for Hartree equation builds on
basic properties of the linear propagator $e^{it\Delta}$.

From the explicit formula $$e^{it\Delta}f(x)=\frac{1}{(4\pi
it)^{d/2}}\int_{\mathbb{R}^d}e^{i|x-y|^2/4t}f(y)dy,$$ we deduce the
standard dispersive inequality
\begin{equation}\label{de}\|e^{it\Delta}f\|_{L^p(\mathbb{R}^d)}
\lesssim\frac{1}{|t|^{d(1/2-1/p)}}\|f\|_{L^{p'}(\mathbb{R}^d)}\end{equation}
for all $t\neq0$ and $2\leq p\leq\infty$.

\begin{lemma}[Kernel estimates, \cite{KTV}] For any $m\geq0$, the kernel of the
linear propagator obeys the following estimates: \begin{equation}
|(P_Ne^{it\Delta})(x,y)|\lesssim_m \left\{
\aligned|t|^{-d/2}\quad\quad\quad\quad\quad&:|x-y|\sim Nt\\
\frac{N^d}{|N^2t|^m\langle N|x-y|\rangle^m}&:\text{otherwise}
\endaligned
\right.
\end{equation}
for $|t|\geq N^{-2}$ and
\begin{equation}|(P_Ne^{it\Delta})(x,y)|\lesssim_m N^d\langle N|x-y|\rangle^{-m}
\end{equation}
for $|t|\leq N^{-2}$.\end{lemma}

\begin{lemma}[Strichartz estimates, \cite{Tao}]\label{stri} Fix $d\geq1$ and call a pair $(q,r)$ admissible
 if $2\leq q,r\leq\infty$, $\frac{2}{q}+\frac{d}{r}=\frac{d}{2}$ and $(q,r,d)\neq (2,\infty,2)$. Then for any
admissible pair $(q,r)$ and $(\tilde{q}, \tilde{r})$, let $I$ be an
interval, let $t_0\in I$, and let $u_0\in L_x^2(\mathbb{R}^d)$ and
$f\in L_{t}^{\tilde{q}'}L_x^{\tilde{r}'}$. Then the function $u$
defined by
\begin{equation}\label{24}u(t):=e^{i(t-t_0)\Delta}u_0-i\int_{t_0}^te^{i(t-t')\Delta}f(t')dt'\end{equation}
obeys the estimate
\begin{equation}\label{tt}\|u\|_{L_t^qL_x^r}\lesssim
\|u_0\|_{L_x^2}+\|f\|_{L_{t}^{\tilde{q}'}L_x^{\tilde{r}'}},\end{equation}
where all spacetime norms are over $I\times\mathbb{R}^d$.
\end{lemma}
\begin{lemma}[Weighted Strichartz, \cite{KVZ}]\label{ws}Let $I$ be an interval, let $t_0\in I$, $u_0\in L_x^2(\mathbb{R}^d)$ and $f\in L_{t}^2L_x^\frac{2d}{d+2}$ be
spherically symmetric. Then the function $u$ defined by (\ref{24})
obeys the estimate
\begin{equation}\label{st}\big\||x|^\frac{2(d-1)}{q}u\big\|_{L_t^qL_x^\frac{2q}{q-4}
(I\times\mathbb{R}^d)}\lesssim
\|u_0\|_{L_x^2(\mathbb{R}^d)}+\big\|f\big\|_{L_{t}^2L_x^\frac{2d}{d+2}(I\times\Bbb
R^d)}
\end{equation}
for all $4\leq q\leq \infty$.
\end{lemma}
\begin{lemma}[Shao's Strichartz estimate, \cite{Shao}]\label{sha} Let $d\geq2$, for $f\in
L_{rad}^2(\mathbb{R}^d)$, we have
\begin{equation}\label{shao}\big\|P_Ne^{it\Delta}f\big\|_{L_{t,x}^q(\Bbb R\times\Bbb R^d)}\lesssim_q
N^{\frac{d}{2}-\frac{d+2}{q}}\|f\|_{L_x^2(\mathbb{R}^d)},
\end{equation}
provided $q>\frac{4d+2}{2d-1}$.
\end{lemma}
\begin{lemma}[Bilinear Strichartz, \cite{Bou2},\cite{CKSTT}] \label{bi}For any spacetime slab
$I\times\mathbb{R}^d$, any $t_0\in I$ and any $M$, $N>0$, we have
\begin{align*}
\big\|(P_{\geq N}u)(P_{\leq M}v)\big\|_{L^2_{t,x}(I\times\Bbb
R^d)}\lesssim_q
&\frac{M^\frac{d-1}{2}}{N^{-\frac{1}{2}}}\Big(\|P_{\geq
N}u(t_0)\|_{L^2(\Bbb R^d)}+\|(i\partial_t+\Delta)P_{\geq
N}u\|_{L_{t}^2L_x^\frac{2d}{d+2}(I\times\Bbb
R^d)}\Big)\\
&\quad\times\Big(\|P_{\leq M}v(t_0)\|_{L^2(\Bbb
R^d)}+\|(i\partial_t+\Delta)P_{\leq
M}v\|_{L_{t}^2L_x^\frac{2d}{d+2}(I\times\Bbb R^d)}\Big)
\end{align*}
for all functions $u$, $v$ on $I\times\Bbb R^d$.
\end{lemma}
\subsection{An in/out decomposition}
We define the projection onto outgoing spherical waves by
$$[P^+f](r)=\frac{1}{2}\int_0^\infty r^{\frac{2-d}{2}}H^{(1)}_{\frac{d-2}{2}}(kr)\hat{f}(k)k^\frac{d}{2}dk$$
and the projection onto incoming spherical waves by
$$[P^-f](r)=\frac{1}{2}\int_0^\infty r^{\frac{2-d}{2}}H^{(2)}_{\frac{d-2}{2}}(kr)\hat{f}(k)k^\frac{d}{2}dk,$$
where $H^{(1)}_{\frac{d-2}{2}}$ denotes the Hankel function of the
first kind with order $\frac{d-2}{2}$ and $H^{(2)}_{\frac{d-2}{2}}$
denotes the Hankel function of the second kind with order
$\frac{d-2}{2}$. We will write $P^\pm_N$ for the product $P^\pm
P_N$.
\begin{lemma}[Kernel estimates, \cite{KVZ}] For $|x|\gtrsim N^{-1}$ and $t\gtrsim
N^{-2}$, the integral kernel obeys
\begin{equation}
\big|[P_N^\pm e^{\mp it\Delta}](x,y)\big|\lesssim \left\{ \aligned
(|x||y|)^{-\frac{d-1}{2}}|t|^{-\frac{1}{2}}\quad\space\quad\quad\quad\quad\quad\quad\quad\quad\quad\quad&: |y|-|x|\sim Nt\\
\frac{N^d}{(N|x|)^{\frac{d-1}{2}}\langle
N|y|\rangle^{\frac{d-1}{2}}}\langle
N^2t+N|x|-N|y|\rangle^{-m}&:\text{otherwise}
\endaligned
\right.
\end{equation}
for any $m\geq0$. For $|x|\gtrsim N^{-1}$ and $|t|\lesssim N^{-2}$,
the integral kernel obeys
\begin{equation*}\big|[P_N^\pm e^{\mp it\Delta}](x,y)\big|\lesssim\frac{N^d}{(N|x|)^{\frac{d-1}{2}}\langle
N|y|\rangle^{\frac{d-1}{2}}}\langle N|x|-N|y|\rangle^{-m}
\end{equation*}
for any $m\geq0$.
\end{lemma}
\renewcommand{\labelenumi}{{\rm(}\roman{enumi}{\rm)}}
\begin{lemma}[Properties of $P^\pm$, \cite{KVZ}]We have
\begin{enumerate}
\item $P^++P^-$ acts as the identity on $L^2_{rad}(\mathbb{R}^d)$.
\item Fix $N>0$. For any spherically symmetric function $f\in
L^2_{x}(\mathbb{R}^d)$, $$\big\|P^\pm P_{\geq
N}f\big\|_{L_x^2(|x|\geq\frac{1}{100}N^{-1})}\lesssim\|f\|_{L^2_{x}(\mathbb{R}^d)}$$
with an $N$-independent constant.
\end{enumerate}
\end{lemma}
\section{Stability and concentration compactness}
\begin{lemma}[Stability] For every $A>0$ and $\varepsilon>0$, there
exists $\delta>0$ with the following property: if $u: I\times\Bbb
R^d\rightarrow\Bbb C$ approximately solves (\ref{har}) in the sense
that
\begin{equation*}iu_t+\Delta
u-F(u)=e,
\end{equation*}
with
\begin{equation*}\|e\|_{L_t^2L_x^\frac{2d}{d+2}(I\times\Bbb R^d)}\leq\delta
\end{equation*}
and obeys
\begin{equation*}\|u\|_{L_t^6L_x^\frac{6d}{3d-2}(I\times\Bbb R^d)}\leq A,
\end{equation*}
and $t_0\in I$ and $v_0\in L_x^2(\Bbb R^d)$ are such that
\begin{equation*}\|u(t_0)-v_0\|_{L_x^2}\leq\delta,
\end{equation*}
then there exists a solution $v: I\times\Bbb R^d\rightarrow\Bbb C$
to (\ref{har}) with $v(t_0)=v_0$ such that
\begin{equation*}\|u-v\|_{L_t^6L_x^\frac{6d}{3d-2}(I\times\Bbb R^d)}\leq\varepsilon.
\end{equation*}
In particular, by the Strichartz inequality,
\begin{equation*}\|u-v\|_{L_t^\infty L_x^2(I\times\Bbb R^d)}\lesssim
\delta+\varepsilon A^2.
\end{equation*}
\end{lemma}
\begin{proof}We first establish this claim when $A$ is sufficiently
small depending on $d$. Let $v: I'\times\Bbb R^d\rightarrow\Bbb C$
be the maximal-lifespan solution with initial data $v(t_0)$. Writing
$v=u+w$ on the interval $I'':=I\cap I'$, then $w$ satisfies
\begin{equation*}iw_t+\Delta w=F(u+w)-F(u)-e
\end{equation*}
with
$$\big\|e^{i(t-t_0)\Delta}w(t_0)\big\|_{L_t^6L_x^\frac{6d}{3d-2}(I''\times\Bbb
R^d)}\leq C_d'\delta.$$ Let
$X:=\|w\|_{L_t^6L_x^\frac{6d}{3d-2}(I''\times\Bbb R^d)}$, then by
Lemma \ref{stri}, we have
\begin{align*}X\leq& C_d'\delta+C_d''\Big(\|F(u+w)-F(u)\|_{L_t^2L_x^\frac{2d}{d+2}(I''\times\Bbb
R^d)}+\delta\Big)\\
\leq &\tilde{C}_d(A^2X+AX^2+X^3+\delta)
\end{align*}
where $\tilde{C}_d$ depends only on $d$. If $A$ is sufficiently
small depending on $d$ and $\delta$ is sufficiently small depending
on $\varepsilon$ and $d$, then the standard continuity argument
gives $X\leq\varepsilon$. If $A$ is large, we can iterate the case
when $A$ is small (shrinking $\delta$, $\varepsilon$ repeatedly)
after a subdivision of the time interval.
\end{proof}

We now need a key concentration-compactness result. The
concentration compactness  principle was first introduced by F.
Merle, L. Vega \cite{mv} and Bahouri, P. Gerard \cite{BG} to study
nonlinear Schr\"odinger equations. The idea was further developed by
S. Keraani \cite{Ker2}. The results of \cite{mv} and \cite{Ker2}
were extended to higher dimensions by P. Begout and A. Vargas
\cite{BV}. Because the solution of the free Schr\"{o}dinger equation
is still a solution under the action of linear propagator
$e^{it_0\Delta}$, we will need to enlarge the group $G$ to contain
this linear propagator.

\begin{definition}[Enlarged group] For any phase $\theta\in\Bbb R/2\pi\Bbb
Z$, position $x_0\in\Bbb R^d$, frequency $\xi_0\in\Bbb R^d$, scaling
parameter $\lambda>0$, and time $t_0$, we define the unitary
transformation $g_{\theta, x_0, \xi_0, \lambda, t_0}: L_x^2(\Bbb
R^d)\rightarrow L_x^2(\Bbb R^d)$ by the formula
\begin{equation*}g_{\theta, x_0, \xi_0, \lambda, t_0}=g_{\theta, x_0, \xi_0,
\lambda}e^{it_0\Delta}.
\end{equation*}
Let $G'$ be the collection of such transformations. In particularly,
we denote by $G_{rad}'$ the collection of all the transformation
with $x_0=\xi_0=0$. We also let $G'$ act on global spacetime
functions $u: \Bbb R\times\Bbb R^d\rightarrow\Bbb C$ by defining
\begin{equation*}T_{g_{\theta, \xi_0, x_0, \lambda,
t_0}}u(t,x):=\frac{1}{\lambda^\frac{d}{2}}e^{i\theta}e^{ix\cdot\xi_0}e^{-it|\xi_0|^2}
(e^{it_0\Delta}u)\Big(\frac{t}{\lambda^2},\frac{x-x_0-2\xi_0t}{\lambda}\Big).
\end{equation*}
\end{definition}
\begin{definition}For any two sequences $g_n$, $g_n'$ in $G'$, we say that $g_n$
and $g_n'$ are {\it asymptotically orthogonal} if $(g_n)^{-1}g_n'$
diverges to infinity in $G'$. More explicitly, if $g_n=g_{\theta_n,
\xi_n, x_n, \lambda_n, t_n}$ and $g_n'=g_{\theta_n', \xi_n', x_n',
\lambda_n', t_n'}$, then this asymptotic orthogonality is equivalent
to
\begin{equation*}
\lim_{n\rightarrow\infty}\bigg(\frac{\lambda_n}{\lambda_n'}+\frac{\lambda_n'}{\lambda_n}
+|t_n\lambda_n^2-t_n'(\lambda_n')^2|+|\xi_n-\xi_n'|+|x_n-x_n'|\bigg)=+\infty.
\end{equation*}
\end{definition}
Careful computation shows that if $g_n$ and $g_n'$ are
asymptotically orthogonal, then
\begin{equation*}\lim_{n\rightarrow\infty}\langle g_nf,g_n'f'
\rangle_{L_x^2(\Bbb R^d)}=0\ \ \text{for all}\ f, f'\in L_x^2(\Bbb
R^d).
\end{equation*}
\begin{theorem}[Linear profiles, \cite{BV}]
Fix $d$. Let $u_n$, $n=1,2,\cdots$ be a bounded sequence in
$L_{rad}^2(\Bbb R^d)$. Then (after passing to a subsequence if
necessary) there exists a family $\phi^{(j)}$, $j=1,2,\cdots$ of
functions in $L_{rad}^2(\Bbb R^d)$ and group elements $g_n^{(j)}\in
G'_{rad}$ for $j,n=1,2,\cdots$ such that we have the decomposition
\begin{equation}\label{equ14}
u_n=\sum_{j=1}^lg_n^{(j)}\phi^{(j)}+w_n^{(l)}
\end{equation}
for all $l=1,2,\cdots$; here $w_n^l\in L_{rad}^2(\Bbb R^d)$ is such
that its linear evolution has asymptotically vanishing scattering
size:
\begin{equation}\label{equ15}
\lim_{l\rightarrow\infty}\limsup_{n\rightarrow\infty}\|e^{it\Delta}w_n^l\|_{L_t^6L_x^\frac{6d}{3d-2}}=0.
\end{equation}
Moreover, $g_n^{(j)}$ and $g_n^{(j')}\in G_{rad}'$ are
asymptotically orthogonal for any $j\neq j'$, and for any $l\geq1$
we have the mass decoupling property
\begin{equation}\label{equ16}
\lim_{n\rightarrow\infty}\big[M(u_n)-\sum_{j=1}^lM(\phi^{(j)})-M(w_n^l)\big]=0.
\end{equation}
\end{theorem}

For later use, we prove the following lemma:
\begin{lemma}\label{lem32}Let $g_n^{(j)}$, $g_n^{(j')}\in G_{rad}'$ be
asymptotically orthogonal for any $j\neq j'$, then we
 have
 \begin{equation}\Big\|\sum_{j=1}^lg_n^{(j)}\phi^{(j)}\Big\|_{L_t^6L_x^\frac{6d}{3d-2}}^6\leq
 \sum_{j=1}^l\Big\|g_n^{(j)}\phi^{(j)}\Big\|_{L_t^6L_x^\frac{6d}{3d-2}}^6+o(1),\quad\text{as}\quad n\rightarrow\infty.\end{equation}
 \end{lemma}
\begin{proof} Since $d\geq3$, $\frac{4}{3d-2}<1$. So we have
 \begin{align*}
\Big\|\sum_{j=1}^lg_n^{(j)}\phi^{(j)}\Big\|_{L_t^6L_x^\frac{6d}{3d-2}}^6
=&\int\bigg(\int\Big|\sum_{j=1}^lg_n^{(j)}\phi^{(j)}\Big|^2\Big|\sum_{p=1}^lg_n^{(p)}\phi^{(p)}\Big|^\frac{4}{3d-2}dx\bigg)^{\frac{3d-2}{d}}dt\\
=&\int\bigg(\sum_{j=1}^l\int\Big|g_n^{(j)}\phi^{(j)}\Big|^2\Big|\sum_{p=1}^lg_n^{(p)}\phi^{(p)}\Big|^\frac{4}{3d-2}dx\\
&\quad\quad\quad+\sum_{j=1}^l\sum_{k=1\atop k\neq
j}^l\int\Big|g_n^{(j)}\phi^{(j)}g_n^{(k)}\phi^{(k)}\Big|\sum_{p=1}^lg_n^{(p)}\phi^{(p)}\Big|^\frac{4}{3d-2}dx\bigg)^\frac{3d-2}{d}dt\\
\leq&\int\bigg(\sum_{j=1}^l\int\Big|g_n^{(j)}\phi^{(j)}\Big|^\frac{6d}{3d-2}dx+\sum_{j=1}^l\sum_{p=1\atop
p\neq
j}^l\int\Big|g_n^{(j)}\phi^{(j)}\Big|^2\Big|g_n^{(p)}\phi^{(p)}\Big|^\frac{4}{3d-2}dx\\
&\quad\quad\quad+\sum_{j=1}^l\sum_{k=1\atop k\neq
j}^l\int\Big|g_n^{(j)}\phi^{(j)}g_n^{(k)}\phi^{(k)}\Big|\Big|\sum_{p=1}^lg_n^{(p)}\phi^{(p)}\Big|^\frac{4}{3d-2}dx\bigg)^\frac{3d-2}{d}dt\\
:=&\int(A+B+C)^\frac{3d-2}{d}dt.
 \end{align*}
 Without loss of generality, we can assume that all $\phi^{(j)}$ are compactly supported in both $t$ and $x$. By orthogonality, B and C vanish as $n\rightarrow\infty$.
 Thus
 \begin{align*}
&\Big\|\sum_{j=1}^lg_n^{(j)}\phi^{(j)}\Big\|_{L_t^6L_x^\frac{6d}{3d-2}}^6
\leq\int\bigg(\sum_{j=1}^l\int\Big|g_n^{(j)}\phi^{(j)}\Big|^{\frac{6d}{3d-2}}dx\bigg)^\frac{3d-2}{d}dt+o(1).\\
 \end{align*}
Now we consider
 \begin{align*}
&\int\bigg(\sum_{j=1}^l\int\Big|g_n^{(j)}\phi^{(j)}\Big|^{\frac{6d}{3d-2}}dx\bigg)^\frac{3d-2}{d}dt\\
=&\int\sum_{j=1}^l\Big(\int\Big|g_n^{(j)}\phi^{(j)}\Big|^\frac{6d}{3d-2}dx\Big)^2
\Big(\sum_{p=1}^l\int\Big|g_n^{(p)}\phi^{(p)}\Big|^\frac{6d}{3d-2}dx\Big)^{\frac{d-2}{d}}dt\\
&\quad+\sum_{j=1}^l\sum_{k=1\atop k\neq
j}^l\int\Big(\int\Big|g_n^{(j)}\phi^{(j)}\Big|^\frac{6d}{3d-2}dx\Big)\Big(\int\Big|g_n^{(k)}\phi^{(k)}\Big|^\frac{6d}{3d-2}dx\Big)
\Big(\sum_{p=1}^l\int\Big|g_n^{(p)}\phi^{(p)}\Big|^\frac{6d}{3d-2}dx\Big)^\frac{d-2}{d}dt\\
:=&I+II.
 \end{align*}
 We estimate $I$ first.
 \begin{align}
 I\leq\sum_{j=1}^l&\int\bigg(\int\Big|g_n^{(j)}\phi^{(j)}\Big|^\frac{6d}{3d-2}dx\bigg)^\frac{3d-2}{d}dt\nonumber\\
 &\quad+\sum_{j=1}^l\sum_{p=1\atop p\neq
j}^l\int\Big(\int\Big|g_n^{(j)}\phi^{(j)}\Big|^\frac{6d}{3d-2}dx\Big)^2\Big(\int\Big|g_n^{(p)}\phi^{(p)}\Big|^\frac{6d}{3d-2}dx\Big)^\frac{d-2}{d}dt.
 \end{align}
 Note that it does not change the compact support of time to take
 space norm, and
 \begin{align*}
\int\Big|g_n^{(j)}\phi^{(j)}\Big|^\frac{6d}{3d-2}dx=&\frac{1}{(\rho_n^j)^\frac{2d}{3d-2}}\int\Big|\phi^{(j)}\bigg(\frac{t}{(\rho_n^j)^2}-t_n^j,x\bigg)\Big|^\frac{6d}{3d-2}dx\\
:=&\frac{1}{(\rho_n^j)^\frac{2d}{3d-2}}L^j\bigg(\frac{t}{(\rho_n^j)^2}-t_n^j\bigg),
 \end{align*}
 we have
 \begin{align*}
 &\int\Big(\int\Big|g_n^{(j)}\phi^{(j)}\Big|^\frac{6d}{3d-2}dx\Big)^2\Big(\int\Big|g_n^{(k)}\phi^{(k)}\Big|^\frac{6d}{3d-2}dx\Big)^\frac{d-2}{d}dt\\
=&\Big(\frac{\rho_n^j}{\rho_n^k}\Big)^\frac{2d-4}{3d-2}\int|L^j(\tilde{t})|^2\Big|
L^k\Big(\Big(\frac{\rho_n^j}{\rho_n^k}\Big)^2\tilde{t}+\Big(\frac{\rho_n^j}{\rho_n^k}\Big)^2t_n^j-t_n^k\Big)\Big|^{\frac{d-2}{d}}d\tilde{t}\\
 \rightarrow& 0\quad\text{as}\quad n\rightarrow\infty.
 \end{align*}
 We can prove similarly that $II\rightarrow 0$ when n is sufficiently large.
\end{proof}

 \section{Almost periodic solutions}
 For brevity, we write $S_I(u)$ to denote
$\|u\|_{L_t^6L_x^\frac{6d}{3d-2}(I\times\Bbb R^d)}$ in this section.
If $I=\Bbb R$, we write $S_\Bbb R(u)=S(u)$.
 \begin{proposition}\label{pro1}Fix $\mu$ and $d$, and suppose that $m_0$ is
finite. Let $u_n: I_n\times\Bbb R^d\rightarrow\Bbb C$ for
$n=1,2,\cdots$ be a sequence of radial solutions and $t_n\in I_n$ a
sequence of times such that
$$\limsup_{n\rightarrow\infty}M(u_n)=m_0,$$ and
$$\lim_{n\rightarrow\infty}S_{\geq t_n}(u_n)=\lim_{n\rightarrow\infty}S_{\leq t_n}(u_n)=\infty.$$
Then the sequence $G_{rad}u_n(t_n)$ has a subsequence which
converges in $G_{rad}\backslash L_x^2$.
\end{proposition}
\begin{proof} By time translation invariance, we may take $t_n=0$ for all
$n$. Then we have $$\lim_{n\rightarrow\infty}S_{\geq
0}(u_n)=\lim_{n\rightarrow\infty}S_{\leq 0}(u_n)=\infty.$$ We
consider the sequence of $\{u_n(0)\}$. Since $\displaystyle
\limsup_{n\rightarrow\infty}M(u_n(0))= m_0$, we have by
concentration compactness principle that
\begin{equation}\label{equ6}u_n(0)=\sum_{j=1}^lg_n^{(j)}\varphi^{(j)}+w_n^l;\
\ \text{with} \ g_n^{(j)}=h_n^{(j)}e^{it_n^j\Delta}\ \text{and}\
h_n^{(j)}\in G_{rad}.
\end{equation}
Moreover, we have the asymptotic orthogonality:
\begin{equation}\label{equ7}\|u_n(0)\|_{L_2}^2=\sum_{j=1}^l\|\varphi^{(j)}\|_{L_2}^2+\|w_n^l\|_{L_2}^2+o(1),
\end{equation}
and \begin{equation}\label{equ9}
\limsup_{n\rightarrow\infty}S(e^{it\Delta}w_n^l)\rightarrow 0\ \
\text{as}\ \ l\rightarrow\infty.
\end{equation}

{\bf Claim}: For any $\varepsilon>0$, $\displaystyle\sup_j M(\varphi^{(j)})\leq m_0-\varepsilon$ doesn't hold .\\
Otherwise, there exists $\varepsilon_0>0$ such that
$M(\varphi^{(j)})\leq m_0-\varepsilon_0$ for any $j=1,2,\cdots$.
Suppose $v^{(j)}$ is the nonlinear profile associated to
$\varphi^{(j)}$ and depending on the limiting value of $t_n^{(j)}$,
namely,
\begin{enumerate}
\item[$\clubsuit$] If $t_n^j$ is identically zero, $v^{(j)}$ is the
maximal-lifespan solution to (\ref{har}) with initial data
$v^{(j)}(0)=\phi^{(j)}$.
\item[$\clubsuit$] If $t_n^j$ converges to $+\infty$, $v^{(j)}$ is the
maximal-lifespan solution to (\ref{har}) which scatters forward in
time to $e^{it\Delta}\phi^{(j)}$.
\item[$\clubsuit$] If $t_n^j$ converges to $-\infty$, $v^{(j)}$ is the
maximal-lifespan solution to (\ref{har}) which scatters backward in
time to $e^{it\Delta}\phi^{(j)}$.
\end{enumerate}

 Let
\begin{equation}\label{equ10}u_n^{(l)}(t)=\sum_{j=1}^lT_{h_n^{(j)}}\big[v^{(j)}(\cdot+t_n^j)\big](t)+e^{it\Delta}w_n^l,
\end{equation}
then we have
\begin{equation}\label{equ11}\lim_{l\rightarrow+\infty}\lim_{n\rightarrow\infty}S(u_n^{(l)})<+\infty.
\end{equation}
In fact, by Lemma \ref{lem32},
\begin{equation*}\lim_{l\rightarrow\infty}\lim_{n\rightarrow\infty}[S(u_n^{(l)})]^6\leq
\lim_{l\rightarrow\infty}\sum_{j=1}^l[S(v^{(j)})]^6.
\end{equation*}
Meanwhile, by (\ref{equ16}), for any $\epsilon>0$ sufficiently
small, there exists $j_0$ such that
$$\|\varphi^{(j)}\|_{L_2}^2<\epsilon, \ \ \text{for all}\  j>j_0.$$
Note that $T_{h_n^j}$ preserves $L_t^6L_x^\frac{6d}{3d-2}$ norm, by
the small data theory, we conclude that for any $j>j_0$, the maximal
lifespan $I^{(j)}=\Bbb R$ and
$$[S(v^{(j)})]^6\leq
CM(\varphi^{(j)})^3.$$ It follows that
$$\sum_{j>j_0}[S(v^{(j)})]^6\leq \sum_{j>j_0}M(\varphi^{(j)})^3\leq C,$$
where we use the mass decoupling property. For $j\leq j_0$, by the
definition of $m_0$ and the fact that $\|v^{(j)}(0)\|_{L^2}\leq
m_0-\varepsilon_0$, we conclude that the maximal lifespan
$I^{(j)}=\Bbb R$ and $S(v^{(j)})\leq C$ . Therefore, we have
\begin{equation*}
\lim_{l\rightarrow+\infty}\lim_{n\rightarrow\infty}S(u_n^{(l)})<\infty.
\end{equation*}
Meanwhile, by mass decoupling and the fact that $h_n^{(j)}$
preserves mass, we get
\begin{align*}
\lim_{n\rightarrow\infty}M(u_n^{(l)}(0)-u_n(0))=&\lim_{n\rightarrow\infty}M(\sum_{j=1}^l(T_{h_n^{(j)}}[v^{(j)}(\cdot+t_n^j)](0)-g_n^{(j)}\varphi^{(j)}))\\
\leq&\lim_{n\rightarrow\infty}\sum_{j=1}^l M(T_{h_n^{(j)}}[v^{(j)}(\cdot+t_n^j)](0)-g_n^{(j)}\varphi^{(j)})\\
=&\lim_{n\rightarrow\infty}\sum_{j=1}^lM(h_n^{(j)}[v^{(j)}(t_n^j)]-h_n^{(j)}e^{it_n^j\Delta}\varphi^{(j)})\\
=&\lim_{n\rightarrow\infty}\sum_{j=1}^lM(v^{(j)}(t_n^j)-e^{it_n^j\Delta}\varphi^{(j)})=0,
\end{align*}
where the last inequality follows from the definition of nonlinear
profile.

Finally, we claim that
\begin{equation}\label{equ13}\lim_{l\rightarrow+\infty}\limsup_{n\rightarrow\infty}
\Big\|(i\partial_t+\Delta)u_n^{(l)}-F(u_n^{(l)})\Big\|_{L_t^2L_x^\frac{2d}{d+2}}=0.
\end{equation}
In fact, write
\begin{equation*}v_n^{(j)}:=T_{h_n^{(j)}}[v^{(j)}(\cdot+t_n^j)].
\end{equation*}
By the definition of $u_n^{(l)}$, we have
$$u_n^{(l)}=\sum_{j=1}^lv_n^{(j)}+e^{it\Delta}w_n^l$$
and $$(i\partial_t+\Delta)u_n^{(l)}=\sum_{j=1}^lF(v_n^{(j)}).$$
Moreover,
\begin{align*}&\Big\|(i\partial_t+\Delta)u_n^{(l)}-F(u_n^{(l)})\Big\|_{L_t^2L_x^\frac{2d}{d+2}}\\
\leq&\Big\|F(u_n^{(l)}-e^{it\Delta}w_n^{(l)})-F(u_n^{(l)})\Big\|_{L_t^2L_x^\frac{2d}{d+2}}+\Big\|\sum_{j=1}^lF(v_n^{(j)})-F(\sum_{j=1}^lv_n^{(j)})\Big\|_{L_t^2L_x^\frac{2d}{d+2}}.
\end{align*}
So it suffices to prove \begin{equation}\label{equ17}
\lim_{l\rightarrow+\infty}\limsup_{n\rightarrow\infty}\Big\|F(u_n^{(l)}-e^{it\Delta}w_n^{(l)})-F(u_n^{(l)})\Big\|_{L_t^2L_x^\frac{2d}{d+2}}=0
\end{equation} and
\begin{equation}\label{equ18}
\lim_{n\rightarrow\infty}\Big\|\sum_{j=1}^lF(v_n^{(j)})-F(\sum_{j=1}^lv_n^{(j)})\Big\|_{L_t^2L_x^\frac{2d}{d+2}}=0.
\end{equation}
By Hardy-Littlewood-Sobolev inequality, we get
\begin{align*}
&\Big\|F(u_n^{(l)}-e^{it\Delta}w_n^{(l)})-F(u_n^{(l)})\Big\|_{L_t^2L_x^\frac{2d}{d+2}}\\
\lesssim&\|e^{it\Delta}w_n^{(l)}\|_{L_t^6L_x^\frac{6d}{3d-2}}\|u_n^{(l)}\|_{L_t^6L_x^\frac{6d}{3d-2}}^2
+\|e^{it\Delta}w_n^{(l)}\|_{L_t^6L_x^\frac{6d}{3d-2}}^2\|u_n^{(l)}\|_{L_t^6L_x^\frac{6d}{3d-2}}
+\|e^{it\Delta}w_n^{(l)}\|_{L_t^6L_x^\frac{6d}{3d-2}}^3.
\end{align*}
Thus (\ref{equ17}) follows from (\ref{equ11}) and (\ref{equ9}). For
(\ref{equ18}), by Minkowski inequality, we have
\begin{align*}&\Big\|\sum_{j=1}^lF(v_n^{(j)})-F(\sum_{j=1}^lv_n^{(j)})\Big\|_{L_t^2L_x^\frac{2d}{d+2}}\\
\lesssim& \sum_{j_1\neq
j_2}\Big\|\big(V*(v_n^{(j_1)}v_n^{(j_2)})\big)v_n^{(j_3)}\Big\|_{L_t^2L_x^\frac{2d}{d+2}}+\sum_{j_1\neq
j_2}\Big\|\big(V*|v_n^{(j_1)}|^2\big)v_n^{(j_2)}\Big\|_{L_t^2L_x^\frac{2d}{d+2}}.
\end{align*}
Since for any $j=1,2,\cdots, l$, $v_n^{(j)}$ is the radial solution
of (\ref{har}) with data $\varphi^{(j)}$,
$\|\varphi^{(j)}\|_{L^2}\leq m_0-\varepsilon$, it follows from mass
conservation and the definition of $m_0$ that
\begin{equation}\label{521}\|v_n^{(j)}\|_{L_t^6L_x^\frac{6d}{3d-2}}\leq C.
\end{equation} Therefore, this together with orthogonality yields
that
\begin{align*}
\Big\|\big(V*(v_n^{(j_1)}v_n^{(j_2)})\big)v_n^{(j_3)}\Big\|_{L_t^2L_x^\frac{2d}{d+2}}
\lesssim&\Big\|v_n^{(j_1)}v_n^{(j_2)}\Big\|_{L_t^3L_x^\frac{3d}{3d-2}}\big\|v_n^{(j_3)}\big\|_{L_t^6L_x^\frac{6d}{3d-2}}\\
\lesssim&\big\|v_n^{(j_1)}v_n^{(j_2)}\big\|_{L_t^3L_x^\frac{3d}{3d-2}}\rightarrow0,
\quad\quad \text{as}\quad n\rightarrow\infty.
\end{align*}
On the other hand, note that $v^{(j)}$ is radial, $h_n^{(j)}\in
G_{rad}$ and that the orthogonality must be possessed by time
variable, we have
\begin{equation*}\Big\|\big(V*|v_n^{(j_1)}|^2\big)v_n^{(j_2)}\Big\|_{L_t^2L_x^\frac{2d}{d+2}}\rightarrow0, \ \ \text{as}\ n\rightarrow\infty.
\end{equation*}
Thus, (\ref{equ18}) follows. At last, by stability, we conclude that
$\displaystyle\lim_{n\rightarrow\infty}S(u_n)<\infty$. This
contradicts the hypothesis.

From the above claim, it follows that $l=1$. So
$u_n(0)=h_ne^{it_n\Delta}\varphi+w_n$ with $M(\varphi)=m_0$. Thus
$M(w_n)\rightarrow0$, which implies that
$S(e^{it\Delta}w_n)\rightarrow0$ as $n\rightarrow\infty$. Without
loss of generality, we may take $h_n$ to be identity. If
$t_n\rightarrow0$, then $u_n(0)\rightarrow\varphi$. Thus
$G_{rad}u_n(0)\rightarrow G_{rad}\varphi$ in $G\backslash L_x^2$. It
suffices to consider the case of $t_n\rightarrow\pm \infty$. We only
consider the case of $t_n\rightarrow+\infty$, the other case is
similar. Then we have
\begin{align*}\lim_{n\rightarrow\infty}S_{\geq0}(e^{it\Delta}u_n(0))&=\lim_{n\rightarrow\infty}S_{\geq0}\big(e^{it\Delta}h_ne^{it_n\Delta
}\varphi\big)=\lim_{n\rightarrow\infty}S_{\geq0}\big(T_{h_n}e^{it\Delta}e^{it_n\Delta
}\varphi\big)\\
&=\lim_{n\rightarrow\infty}S_{\geq0}\big(e^{i(t+t_n)\Delta
}\varphi\big)=\lim_{n\rightarrow\infty}S_{\geq
t_n}(e^{it\Delta}\varphi).
\end{align*}
Since $S(e^{it\Delta}\varphi)<+\infty$, we have
$$\lim_{n\rightarrow\infty}S_{\geq
t_n}(e^{it\Delta}\varphi)=0.$$ By stability again, we have
$\displaystyle\lim_{n\rightarrow\infty}S_{\geq 0}(u_n)=0$ and we
reach a contradiction.
\end{proof}
 {\it Proof of Theorem \ref{3sc}.} We will only prove the first half of the theorem because the proof
of the second half is identical with that of \cite{KTV}, which
relies only the structure of group $G_{rad}$, pseudo-conformal
invariance of (\ref{har}) and is combinatorial.

Suppose Theorem \ref{main} failed, then there exists a sequence of
radial solutions $u_n$ of (\ref{har}) with $M(u_n)\leq m_0$ and
$\displaystyle\lim_{n\rightarrow\infty}S(u_n)=+\infty.$ Suppose
$u_n$ is maximal lifespan solutions, then there exists $t_n\in I_n$
such that $\displaystyle\lim_{n\rightarrow\infty}S_{\geq
t_n}(u_n)=\lim_{n\rightarrow\infty}S_{\leq t_n}(u_n)=\infty$. By
translation invariance, we may take $t_n$ to be zero. From
Proposition \ref{pro1}, it follows that there exists $u_0\in L^2$
such that $G_{rad}u_n(0)\rightarrow G_{rad}u_0$, namely,
$g_nu_n(0)\rightarrow u_0$ for some $g_n\in G_{rad}$. Without loss
of generality, we may assume that $g_n$ is identity. Thus
$u_n(0)\rightarrow u_0$. Moreover, $M(u_0)\leq m_0$.

Let $u$ be the radial solution of (\ref{har}) with initial data
$u_0$, then $u$ blows up both forward and backward in time. In fact,
if $u$ doesn't blow up forward in time, we have
$S_{\geq0}(u)<+\infty$. By stability, we have
$\displaystyle\limsup_{n\rightarrow+\infty}S_{\geq0}(u_n)<+\infty$.
This contradicts the asymptotically blow-up. Similarly, we can prove
that $u$ blows up backward in time. By the definition of $m_0$, we
have $m_0\leq M(u_0)$. Thus we have $M(u_0)=m_0$.

Now we consider any sequence $G_{rad}u(t_n')$ for $t_n'\in I_n$.
Since $u$ blows up forward and backward in time, we have $S_{\geq
t_n'}(u)=S_{\leq t_n'}(u)=\infty$. Then by Proposition \ref{pro1},
we have $G_{rad}u(t_n')\rightarrow G_{rad}u_0$ (up to subsequence).
Therefore, $\{G_{rad}u(t), t\in I\}$ is precompact in
$G_{rad}\backslash L^2_x$.
\begin{proposition}[Spacetime bound]\label{sb} Let $u$ be a non-zero solution to
(\ref{har}) with lifespan $I$, which is almost periodic modulo $G$
with frequency scale function $N: I\rightarrow\mathbb{R}^+$. If $J$
is any subinterval of $I$, then
\begin{equation}\label{spacetime}\int_J N(t)^2dt\lesssim_u
\int_J\Big(\int_{\mathbb{R}^d}|u(t,x)|^\frac{6d}{3d-2}dx\Big)^{\frac{3d-2}{d}}dt\lesssim_u
1+\int_J N(t)^2dt.
\end{equation}
\end{proposition}
\begin{proof} We first prove that
\begin{equation}\label{21}\int_J\Big(\int_{\mathbb{R}^d}|u(t,x)|^\frac{6d}{3d-2}dx\Big)^{\frac{3d-2}{d}}dt\lesssim_u
1+\int_J N(t)^2dt.\end{equation} Let $0<\eta<1$ to be chosen
momentarily and partition $J$ into subintervals $I_j$ so that
\begin{equation}\eta/2\leq\int_{I_j}N(t)^2dt\leq\eta,
\end{equation}
this requires at most $\eta^{-1}\times\text{RHS}(\ref{spacetime})$
intervals. For each $j$, we may choose $t_j\in I_j$ so that
\begin{equation}
N(t_j)^2|I_j|\leq2\eta.
\end{equation}
By Strichartz estimates, we have the following estimates on the
spacetime slab $I_j\times{\Bbb R}^d$
\begin{align*}
\|u\|_{L_t^6L_x^\frac{6d}{3d-2}}\lesssim&\Big\|e^{i(t-t_j)\Delta}u(t_j)\Big\|_{L_t^6L_x^\frac{6d}{3d-2}}+\|u\|_{L_t^6L_x^\frac{6d}{3d-2}}^3\\
\lesssim&\|u_{>N_0}(t_j)\|_{L_x^2}+\Big\|e^{i(t-t_j)\Delta}u_{<N_0}(t_j)\Big\|_{L_t^6L_x^\frac{6d}{3d-2}}+\|u\|_{L_t^6L_x^\frac{6d}{3d-2}}^3\\
\lesssim&\|u_{>N_0}(t_j)\|_{L_x^2}+|I_j|^\frac{1}{6}N_0^\frac{1}{3}\|u(t_j)\|_{L_x^2}+\|u\|_{L_t^6L_x^\frac{6d}{3d-2}}^3\\
\lesssim&\|u_{>N_0}(t_j)\|_{L_x^2}+\eta^\frac{1}{6}\|u(t_j)\|_{L_x^2}+\|u\|_{L_t^6L_x^\frac{6d}{3d-2}}^3.
\end{align*}
Choosing $N_0$ as a large multiple of $N(t_j)$ and using Definition
\ref{aps}, one can make the first term arbitrary small. Choosing
$\eta$ sufficiently small depending on $M(u)$, one may also render
the second term arbitrarily small. Thus by the bootstrap argument we
obtain$$\int_{I_j}\Big(\int_{\mathbb{R}^d}|u(t,x)|^\frac{6d}{3d-2}dx\Big)^{\frac{3d-2}{d}}dt\lesssim\eta.$$
So (\ref{21}) follows if we use the bound on the number of intervals
$I_j$.

Now we prove
\begin{equation}\label{23}\int_J\Big(\int_{\mathbb{R}^d}|u(t,x)|^\frac{6d}{3d-2}dx\Big)^{\frac{3d-2}{d}}dt\gtrsim_u
\int_J N(t)^2dt.\end{equation} Using Definition \ref{aps} and
choosing $\eta$ sufficiently small depending on $M(u)$, we can
guarantee that
\begin{equation}\label{22}\int_{|x-x(t)|\leq
C(\eta)N(t)^{-1}}|u(t,x)|^2dx\gtrsim_u1.\end{equation} By H\"older
inequality, we get
\begin{align*}\int_{{\Bbb
R}^d}|u(t,x)|^\frac{6d}{3d-2}dx \gtrsim\Big(\int_{|x-x(t)|\leq
C(\eta)N(t)^{-1}}|u(t,x)|^2dx\Big)^\frac{3d}{3d-2}N(t)^\frac{2d}{3d-2}.
\end{align*}
Using (\ref{22}) and integrating over $J$ we derive (\ref{23}).
\end{proof}
\section{The self-similar solutions}

 This section is devoted to proving Theorem {\ref{self-similar}}. Let $u$ be as in Theorem \ref{self-similar}. For any
 $A>0$, we define\begin{equation}\label{m}\mathcal
 {M}(A):=\sup_{T>0}\big\|u_{>AT^{-1/2}}(T)\big\|_{L_x^2(\mathbb{R}^d)},\qquad \qquad \qquad \end{equation}
 \begin{equation}\label{s}\mathcal
 {S}(A):=\sup_{T>0}\big\|u_{>AT^{-1/2}}(T)\big\|_{L_t^6L_x^\frac{6d}{3d-2}([T,2T]\times\mathbb{R}^d)},\end{equation}
 \begin{equation}\label{n}\mathcal
 {N}(A):=\sup_{T>0}\big\|P_{>AT^{-1/2}}F(u)\big\|_{L_{t}^2L_x^\frac{2d}{d+2}([T,2T]\times\mathbb{R}^d)},\end{equation}
 where $u_{>AT^{-1/2}}(T)=P_{>AT^{-1/2}}u(T)$. To prove Theorem \ref{self-similar}, it suffices to show that for
 every $s>0$, $$\mathcal{M}(A)\lesssim_{s,u}A^{-s}$$ whenever
 $A$ is sufficiently large depending on $u$ and $s$. From mass
 conservation, Proposition \ref{sb}, self-similarity and
 Hardy-Littlewood-Sobolev inequality, we have
 \begin{equation}\label{41}\mathcal{M}(A)+\mathcal{S}(A)+\mathcal{N}(A)\lesssim_u 1\end{equation}
 for all $A>0$. From Strichartz estimates, we also see that
 \begin{equation}\label{42}
\mathcal{S}(A)\lesssim\mathcal{M}(A)+\mathcal{N}(A)
 \end{equation}
 for all $A>0$. A similar application of Strichartz estimates shows that for
 any admissible pair $(q,r)$,
 \begin{equation}\label{43}\|u\|_{L_t^qL_x^r([T, 2T]\times\Bbb
 R^d)}\lesssim_u 1
\end{equation}
for all $T>0$.
\begin{lemma}[Nonlinear estimate]\label{NE} For all $A>100$, we have
$$\mathcal{N}(A)\lesssim_u \mathcal{S}(\frac{A}{8})\mathcal{M}
(\sqrt{A})+A^{-\frac{1}{2d+6}}[\mathcal{M}(\frac{A}{8})+\mathcal{N}(\frac{A}{8})].$$
\end{lemma}
\begin{proof} It suffices to prove that
\begin{equation}\label{44}
\Big\|P_{>AT^{-\frac{1}{2}}}\big(F(u)\big)\Big\|_{L_t^2L_x^\frac{2d}{d+2}}
\lesssim_u\mathcal{S}(\frac{A}{8})\mathcal{M}
(\sqrt{A})+A^{-\frac{1}{2d+6}}[\mathcal{M}(\frac{A}{8})+\mathcal{N}(\frac{A}{8})]
\end{equation}
for arbitrary $T>0$. To do this, we decompose $u$ as
$$u=u_{\geq \frac{1}{8}AT^{-\frac{1}{2}}}+u_{\frac{1}{8}AT^{-\frac{1}{2}}>\cdot\geq \sqrt{A}T^{-\frac{1}{2}}}+u_{<\sqrt{A}T^{-\frac{1}{2}}}.$$
Then any term in the resulting expansion of
$P_{>AT^{-\frac{1}{2}}}\big(F(u)\big)$ that does not contain the
factor of $u_{\geq \frac{1}{8}AT^{-\frac{1}{2}}}$ vanishes.

Consider the terms which contain at least one factor of
$u_{\frac{1}{8}AT^{-\frac{1}{2}}>\cdot\geq
\sqrt{A}T^{-\frac{1}{2}}}$. The term which contains three factors
of $u_{>\frac{1}{8}AT^{-\frac{1}{2}}}$ or which contains two
factors of $u_{>\frac{1}{8}AT^{-\frac{1}{2}}}$ and one factor of
$u_{<\sqrt{A}T^{-\frac{1}{2}}}$ can be estimated similarly. By
H\"older inequality, Hardy-Littlewood-Sobolev inequality,
(\ref{m}), (\ref{s}) and (\ref{43}), we have
\begin{align*}
\Big\|V*\big(u_{\geq
\frac{1}{8}AT^{-\frac{1}{2}}}&u_{\frac{1}{8}AT^{-\frac{1}{2}}>\cdot\geq
\sqrt{A}T^{-\frac{1}{2}}}\big)u\Big\|_{L_t^2L_x^\frac{2d}{d+2}([T,2T]\times{\Bbb R}^d)}\\
\lesssim&\Big\|V*\big(u_{\geq
\frac{1}{8}AT^{-\frac{1}{2}}}u_{\frac{1}{8}AT^{-\frac{1}{2}}>\cdot\geq
\sqrt{A}T^{-\frac{1}{2}}}\big)\Big\|_{L_t^6L_x^\frac{3d}{5}([T,2T]\times{\Bbb
R}^d)}\|u\|_{L_t^3L_x^\frac{6d}{3d-4}([T,2T]\times{\Bbb R}^d)}\\
\lesssim&\Big\|u_{\geq
\frac{1}{8}AT^{-\frac{1}{2}}}u_{\frac{1}{8}AT^{-\frac{1}{2}}>\cdot\geq
\sqrt{A}T^{-\frac{1}{2}}}\Big\|_{L_t^6L_x^\frac{3d}{3d-1}([T,2T]\times{\Bbb R}^d)}\\
\lesssim&\big\|u_{>\frac{1}{8}AT^{-\frac{1}{2}}}\big\|_{L_t^6L_x^\frac{6d}{3d-2}[T,2T]\times{\Bbb
R}^d)}\big\|u_{\frac{1}{8}AT^{-\frac{1}{2}}>\cdot\geq
\sqrt{A}T^{-\frac{1}{2}}}\big\|_{L_t^\infty
L_x^2([T,2T]\times{\Bbb R}^d)}\\
\lesssim&\mathcal{S}(\frac{A}{8})\mathcal{M}(\sqrt{A}),
\end{align*}
and
\begin{align*}
\Big\|V*\big(u_{\geq
\frac{1}{8}AT^{-\frac{1}{2}}}&u\big)u_{\frac{1}{8}AT^{-\frac{1}{2}}>\cdot\geq
\sqrt{A}T^{-\frac{1}{2}}}\Big\|_{L_t^2L_x^\frac{2d}{d+2}([T,2T]\times{\Bbb R}^d)}\\
\lesssim&\Big\|V*\big(u_{\geq
\frac{1}{8}AT^{-\frac{1}{2}}}u\big)\Big\|_{L_t^2L_x^d([T,2T]\times{\Bbb
R}^d)}\big\|u_{\frac{1}{8}AT^{-\frac{1}{2}}>\cdot\geq
\sqrt{A}T^{-\frac{1}{2}}}\big\|_{L_t^\infty
L_x^2([T,2T]\times{\Bbb R}^d)}\\
\lesssim&\Big\|u_{\geq
\frac{1}{8}AT^{-\frac{1}{2}}}u\Big\|_{L_t^2L_x^\frac{d}{d-1}([T,2T]\times{\Bbb R}^d)}\mathcal{M}(\sqrt{A})\\
\lesssim&\big \|u_{\geq
\frac{1}{8}AT^{-\frac{1}{2}}}\big\|_{L_t^6L_x^\frac{6d}{3d-2}([T,2T]\times{\Bbb
R}^d)}
\big\|u\big\|_{L_t^3L_x^\frac{6d}{3d-4}([T,2T]\times{\Bbb R}^d)}\mathcal{M}(\sqrt{A})\\
\lesssim&\mathcal{S}(\frac{A}{8})\mathcal{M}(\sqrt{A}).
\end{align*}
Similarly, we have
$$\Big\|V*\big(u_{\frac{1}{8}AT^{-\frac{1}{2}}>\cdot\geq
\sqrt{A}T^{-\frac{1}{2}}}u\big)u_{\geq
\frac{1}{8}AT^{-\frac{1}{2}}}\Big\|_{L_t^2L_x^\frac{2d}{d+2}([T,2T]\times{\Bbb
R}^d)}\lesssim \mathcal{S}(\frac{A}{8})\mathcal{M}(\sqrt{A}).$$

Finally, we consider the terms with one factor of $u_{\geq
\frac{1}{8}AT^{-\frac{1}{2}}}$ and two factors of
$u_{<\sqrt{A}T^{-\frac{1}{2}}}$. First, by H\"older inequality,
Lemma \ref{Bern}, Lemma \ref{bi}, (\ref{m}), (\ref{n}) and
(\ref{43}), we have
\begin{align*}
\Big\|V*\big(u_{\geq
\frac{1}{8}AT^{-\frac{1}{2}}}&u_{<\sqrt{A}T^{-\frac{1}{2}}}\big)u_{<\sqrt{A}T^{-\frac{1}{2}}}\Big\|_{L_t^2L_x^\frac{2d}{d+2}([T,2T]\times{\Bbb R}^d)}\\
\lesssim&\Big\|V*\big(u_{\geq
\frac{1}{8}AT^{-\frac{1}{2}}}u_{<\sqrt{A}T^{-\frac{1}{2}}}\big)\Big\|_{L_t^2L_x^d([T,2T]\times{\Bbb R}^d)}\big\|u_{<\sqrt{A}T^{-\frac{1}{2}}}\big\|_{L_t^\infty L_x^2([T,2T]\times{\Bbb R}^d)}\\
\lesssim&_u\frac{1}{(\frac{1}{8}AT^{-\frac{1}{2}})^\frac{d-2}{2}}\big\|u_{>\frac{1}{8}AT^{-\frac{1}{2}}}u_{<\sqrt{A}T^{-\frac{1}{2}}}\big\|_{L_{t,x}^2}\\
\lesssim&_uA^{-\frac{d-1}{4}}
\Big(\|P_{>\frac{1}{8}AT^{-\frac{1}{2}}}u(t_0)\|_{L_x^2}+\big\|(i\partial_t+\Delta)P_{>\frac{1}{8}AT^{-\frac{1}{2}}}u\big\|_{L_t^2L_x^\frac{2d}{d+2}}\Big)\\
&\quad\quad\times\Big(\|P_{<\sqrt{A}T^{-\frac{1}{2}}}u(t_0)\|_{L_x^2}+\big\|(i\partial_t+\Delta)P_{<\sqrt{A}T^{-\frac{1}{2}}}u\big\|_{L_t^2L_x^\frac{2d}{d+2}}\Big)
\\
\lesssim&_uA^{-\frac{d-1}{4}}\big(\mathcal{M}(\frac{A}{8})+\mathcal{N}(\frac{A}{8})\big).
\end{align*}
Now it suffices to estimate
\begin{align*}
\Big\|\big(V*|u_{<\sqrt{A}T^{-\frac{1}{2}}}|^2\big)u_{>\frac{1}{8}AT^{-\frac{1}{2}}}\Big\|_{L_t^2L_x^\frac{2d}{d+2}([T,2T]\times{\Bbb
R}^d)}.\end{align*}
We divide it into two terms. 
\begin{align*}
&\Big\|\Big(V*|u_{<\sqrt{A}T^{-\frac{1}{2}}}|^2\Big)u_{>\frac{1}{8}AT^{-\frac{1}{2}}}
\Big\|_{L_t^2L_x^\frac{2d}{d+2}([T, 2T]\times\Bbb R^d)}\\
\leq&\Big\|\Big(V*|u_{<\sqrt{A}T^{-\frac{1}{2}}}|^2\Big)P_{>\frac{1}{8}AT^{-\frac{1}{2}}}e^{i(t-T)\Delta}u(T)
\Big\|_{L_t^2L_x^\frac{2d}{d+2}([T, 2T]\times\Bbb
R^d)}\\
&\quad\quad+\Big\|\Big(V*|u_{<\sqrt{A}T^{-\frac{1}{2}}}|^2\Big)\int_T^tP_{>\frac{1}{8}AT^{-\frac{1}{2}}}e^{i(t-t')\Delta}
F(u(t'))dt'\Big\|_{L_t^2L_x^\frac{2d}{d+2}([T, 2T]\times\Bbb
R^d)}\\
:=&I+II.
\end{align*}
Let $$q=\frac{2d+6}{d+1},\ \
s=-\frac{d}{2}+\frac{d+2}{q}=\frac{1}{d+3}>0$$ and $q_1$, $r_1$,
$r_2$ satisfy
$$\frac{1}{q_1}+\frac{1}{q}=\frac{1}{2}, \ \ \frac{1}{r_1}+\frac{1}{q}=\frac{d+2}{2d}, \ \ \frac{1}{r_1}=\frac{1}{r_2}-\frac{d-2+s}{d},$$
which yields that
$$\frac{2}{2q_1}+\frac{d}{2r_2}=\frac{d}{2}.$$
Thus, By H\"older inequality, Sobolev imbedding, Lemma \ref{Bern},
Lemma \ref{sha} and (\ref{43}), we have \begin{align*}
I\leq&\big\|V*|u_{<\sqrt{A}T^{-\frac{1}{2}}}|^2\big\|_{L_t^{q_1}L_x^{r_1}([T,
2T]\times\Bbb
R^d)}\big\|P_{>\frac{1}{8}AT^{-\frac{1}{2}}}e^{i(t-T)\Delta}u(T)\big\|_{L_{t,x}^q([T,
2T]\times\Bbb R^d)}\\
\leq&
C\big\||\nabla|^s\big(|u_{<\sqrt{A}T^{-\frac{1}{2}}}|^2\big)\big\|_{L_t^{q_1}L_x^{r_2}([T,
2T]\times\Bbb
R^d)}\sum_{N>\frac{1}{8}AT^{-\frac{1}{2}}}\big\|P_Ne^{i(t-T)\Delta}u_{>\frac{1}{8}AT^{-\frac{1}{2}}}(T)\big\|_{L_{t,x}^q([T,
2T]\times\Bbb R^d)}\\
\lesssim&
(\sqrt{A}T^{-\frac{1}{2}})^s\|u\|_{L_t^{2q_1}L_x^{2r_2}([T,
2T]\times\Bbb
R^d)}^2\sum_{N>\frac{1}{8}AT^{-\frac{1}{2}}}N^{\frac{d}{2}-\frac{d+2}{q}}\|P_{>\frac{1}{8}AT^{-\frac{1}{2}}}u(T)\|_{L_x^2}\\
\lesssim&_u
(\sqrt{A}T^{-\frac{1}{2}})^s(AT^{-\frac{1}{2}})^{\frac{d}{2}-\frac{d+2}{q}}\mathcal{M}(\frac{A}{8})
\lesssim_u A^{-\frac{1}{2d+6}}\mathcal{M}(\frac{A}{8}).
\end{align*}
Now we estimate $II$. By duality and Lemma \ref{sha}, we get that
\begin{equation}\label{62}
\|\int_{\Bbb R}P_Ne^{-it\Delta}f(t,x)dt\|_{L_x^2}\leq
CN^{\frac{d}{2}-\frac{d+2}{q}}\|f\|_{L_{t,x}^{q'}}.
\end{equation}
Therefore, this together with Lemma \ref{stri} yields that for
$q>\frac{4d+2}{2d-1}$,
\begin{align*}
\Big\|\int_{\Bbb R}P_Ne^{i(t-t')\Delta}&F(t')dt'\Big\|_{L_{t,x}^q}\\
=&\sup_{\|g\|_{L_{t,x}^{q'}}=1}\int_{\Bbb R}\langle\int_{\Bbb
R}P_Ne^{i(t-t')\Delta}F(t')dt', g(t,x)\rangle dt\\
\leq& \sup_{\|g\|_{L_{t,x}^{q'}}=1}\Big\|\int_{\Bbb
R}e^{-it'\Delta}F(t',x)dt'\Big\|_{L_x^2}\Big\|\int_{\Bbb
R}P_Ne^{-it\Delta}g(t,x)dt\Big\|_{L_x^2}\\
\lesssim&
N^{\frac{d}{2}-\frac{d+2}{q}}\|F\|_{L_t^2L_x^\frac{2d}{d+2}}.
\end{align*}
From this and Christ-Kiselev lemma (see \cite{Tao}), we get that
$$\Big\|\int_{t_0}^tP_Ne^{i(t-t')\Delta}P_{>\frac{1}{8}AT^{-\frac{1}{2}}}F(u(t'))dt'\Big\|_{L_{t,x}^q}\leq
CN^{\frac{d}{2}-\frac{2d}{d+2}}\|P_{>\frac{1}{8}AT^{-\frac{1}{2}}}F\|_{L_t^2L_x^\frac{2d}{d+2}}.$$
Similar to the estimate of $I$, we get that
\begin{equation*}
II\lesssim_u A^{-\frac{1}{2d+6}}\mathcal{N}(\frac{A}{8}).
\end{equation*}
\end{proof}
\begin{lemma}[Qualitative decay]\label{qd} We have
\begin{equation}\label{equ43}\lim_{A\rightarrow\infty}\mathcal{M}(A)=
\lim_{A\rightarrow\infty}\mathcal{S}(A)=\lim_{A\rightarrow\infty}\mathcal{N}(A)=0.
\end{equation}
\end{lemma}
\begin{proof}Since $u$ is almost periodic modulo $G_{rad}$ and
self-similar, for any $\eta>0$, there exists $C(\eta)$, such that
$$\int_{|\xi|>C(\eta)T^{-\frac{1}{2}}}|\hat{u}(T,
\xi)|^2d\xi<\eta.$$ In particular, when $A$ is sufficiently large,
we have $$\int_{|\xi|>AT^{-\frac{1}{2}}}|\hat{u}(T,
\xi)|^2d\xi<\eta.$$ Therefore, we get
$$\lim_{A\rightarrow\infty}\mathcal{M}(A)=0.$$ From Lemma \ref{NE}
and (\ref{41}), we have
\begin{align*}\mathcal{N}(A)\leq&
\mathcal{S}(\frac{A}{8})\mathcal{M}(\sqrt{A})+A^{-\frac{1}{2d+6}}\big(\mathcal{M}(\frac{A}{8})+\mathcal{N}(\frac{A}{8})\big)\\
\lesssim&_u \mathcal{M}(\sqrt{A})+A^{-\frac{1}{2d+6}}\rightarrow 0\
\ \text{as}\ \ A\rightarrow\infty.\end{align*} Finally, (\ref{42})
yields
$$\mathcal{S}(A)\lesssim \mathcal{M}(A)+\mathcal{N}(A)\rightarrow 0\ \ \text{as}\ \ A\rightarrow\infty.$$
\end{proof}
\begin{proposition}\label{pro51}Let $0<\eta<1$. Then if $A$ is sufficiently
large depending on $u$ and $\eta$,
\begin{equation}\label{equ41}\mathcal{S}(A)\leq\eta\mathcal{S}(\frac{A}{16})+A^{-\frac{1}{2d+6}}.\end{equation}
In particular, $\mathcal{S}(A)\lesssim_u A^{-\frac{1}{2d+6}}$ for
all $A>0$.
\end{proposition}
\begin{proof} Fix $\eta\in(0,1)$. It suffices to show that
\begin{equation}
\Big\|u_{>AT^{-\frac{1}{2}}}\Big\|_{L_t^6L_x^\frac{6d}{3d-2}[T,2T]\times{\Bbb
R}^d}\leq\eta\mathcal{S}\big(\frac{A}{16}\big)+A^{-\frac{1}{2d+6}}
\end{equation}
for all $T>0$.\\
\indent Fix $T>0$. By the Duhamel formula and then using Lemma
\ref{stri}, we obtain
\begin{align*}
\Big\|u_{>AT^{-\frac{1}{2}}}\Big\|_{L_t^6L_x^\frac{6d}{3d-2}[T,2T]\times{\Bbb
R}^d}\lesssim&\Big\|P_{>AT^{-\frac{1}{2}}}e^{i(t-\frac{T}{2})\Delta}u(\frac{T}{2})\Big\|_{L_t^6L_x^\frac{6d}{3d-2}[T,2T]\times{\Bbb
R}^d}\\
&+\Big\|P_{>AT^{-\frac{1}{2}}}F(u)\Big\|_{L_t^2L_x^\frac{2d}{d+2}([\frac{T}{2},2T]\times\Bbb
R^d)}.
\end{align*}
First, we consider the second term. By definition, we have
$$\big\|P_{>AT^{-\frac{1}{2}}}F(u)\big\|_{L_t^2L_x^\frac{2d}{d+2}([\frac{T}{2},2T]\times\Bbb
R^d)}\lesssim \mathcal{N}(A/2).$$ Using Lemma \ref{NE} and Lemma
\ref{qd}, we derive
that$$\big\|P_{>AT^{-\frac{1}{2}}}F(u)\big\|_{L_t^2L_x^\frac{2d}{d+2}([\frac{T}{2},2T]\times\Bbb
R^d)}\lesssim RHS(\ref{equ41}).$$ Thus the second term is
acceptable.

\begin{center}
\resizebox{5cm}{5cm}{\includegraphics{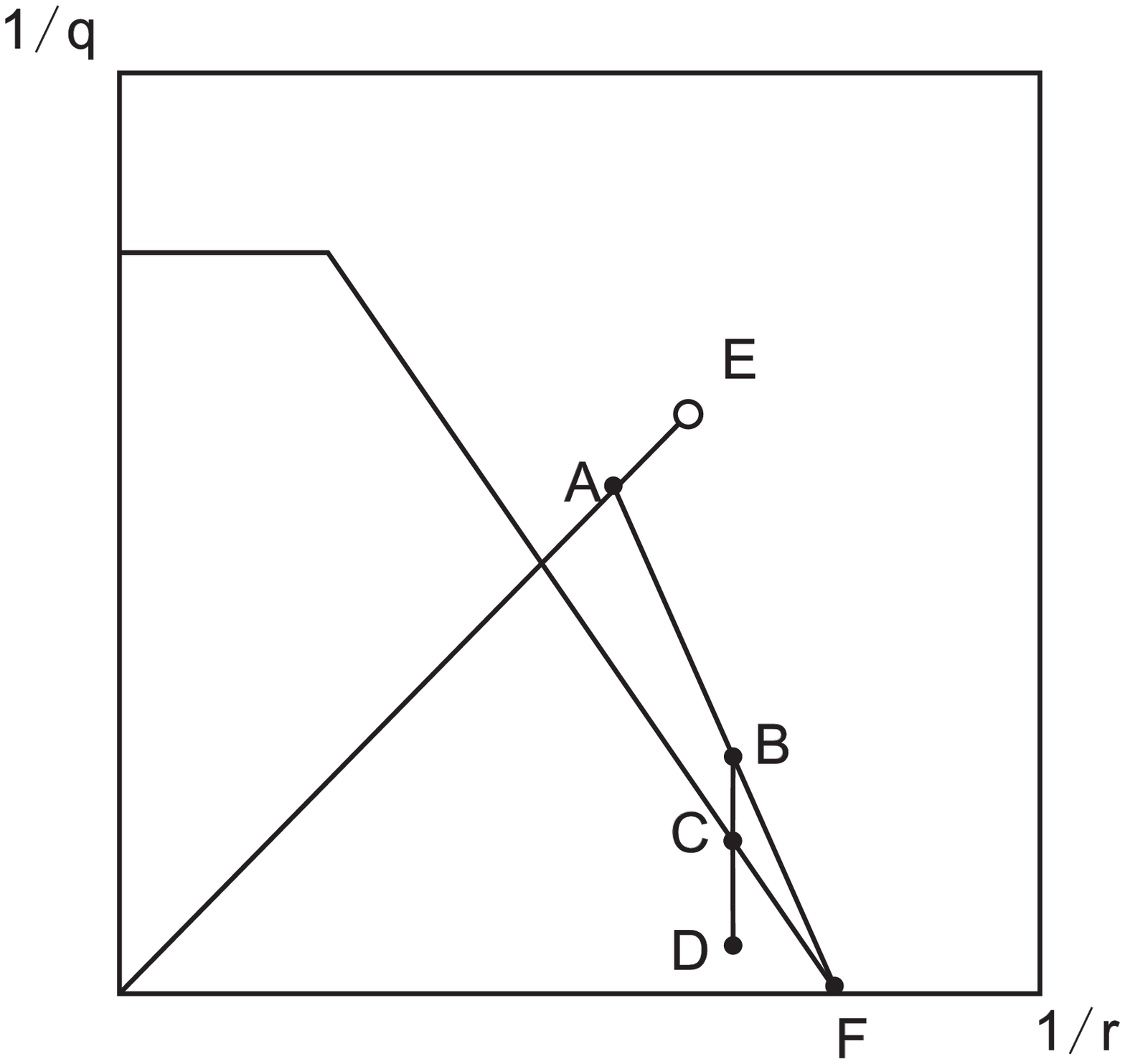}} \\
$A=(\frac{d-1}{2d}, \frac{d-1}{2d})$, $B=(\frac{3d-2}{6d},
\frac{d-1}{3d})$, $C=(\frac{3d-2}{6d}, \frac{1}{6})$,\\
\vspace*{0.3cm} $D=(\frac{3d-2}{6d}, \frac{1}{4d+2})$,
$E=(\frac{2d-1}{4d+2}, \frac{2d-1}{4d+2})$, $F=(\frac{1}{2}, 0)$.\\
\vspace*{0.3cm} {\bf Figure 1. Interpolation game board.}
\end{center}


 Now we consider the first
term (see Figure 1). In fact, we will show that
\begin{equation}\label{equ42}\Big\|P_{>AT^{-\frac{1}{2}}}e^{i(t-\frac{T}{2})\Delta}u(\frac{T}{2})\Big\|_{L_t^6L_x^\frac{6d}{3d-2}[T,2T]\times{\Bbb
R}^d}\lesssim_u A^{-\frac{2(d-1)(d-2)}{3(4d^2-5d-2)}},
\end{equation}
which is acceptable, since
$\frac{2(d-1)(d-2)}{3(4d^2-5d-2)}>\frac{1}{2d+6}$ for all $d\geq3$.
From Shao's Strichartz estimate (\ref{shao}), we have
\begin{equation}\label{45}\Big\|P_Me^{it\Delta}f\Big\|_{L_{t,x}^\frac{2d}{d-1}}\leq
CM^{-\frac{d-2}{2d}}\|f\|_{L_x^2}.
\end{equation} Meanwhile, the Strichartz estimate gives
\begin{equation}\label{46}
\|e^{it\Delta}f\|_{L_t^\infty L_x^2}\leq
C\|f\|_{L_x^2}.\end{equation} Interpolating between (\ref{45}) and
(\ref{46}) with $\theta=2/3$, we have
\begin{equation}\label{47}
\Big\|P_Me^{it\Delta}f\Big\|_{L_t^\frac{3d}{d-1}L_x^\frac{6d}{3d-2}}\leq
CM^{-\frac{d-2}{3d}}\|f\|_{L_x^2}.
\end{equation}
Thus, from the mass conservation, we have
\begin{equation}\label{equ44}
\Big\|P_{BT^{-\frac{1}{2}}}e^{i(t-\frac{T}{2})\Delta}u(\frac{T}{2})\Big\|_{L_t^\frac{3d}{d-1}L_x^\frac{6d}{3d-2}}\lesssim_u
\big(BT^{-\frac{1}{2}}\big)^{-\frac{d-2}{3d}}.
\end{equation}
Using the Duhamel formula, we write
\begin{equation*}
P_{BT^{-\frac{1}{2}}}e^{i(t-\frac{T}{2})\Delta}u(\frac{T}{2})=P_{BT^{-\frac{1}{2}}}e^{i(t-\varepsilon)\Delta}u(\varepsilon)
-i\int_\varepsilon^\frac{T}{2}P_{BT^{-\frac{1}{2}}}e^{i(t-t')\Delta}F(u(t'))dt'
\end{equation*}
for any $\varepsilon>0$. By self-similarity, the first term of RHS
converges strongly to zero in $L_x^2$ as $\varepsilon\rightarrow0$.
By Lemma \ref{Bern}, it also converges to zero in
$L_x^\frac{6d}{3d-2}$. Thus using the dispersive estimate
(\ref{de}), we obtain
\begin{align*}
\big\|P_{BT^{-\frac{1}{2}}}e^{i(t-\frac{T}{2})\Delta}u(\frac{T}{2})
&\big\|_{L_t^{4d+2}L_x^\frac{6d}{3d-2}([T,2T]\times\Bbb R^d)}\\
\lesssim& \Big\|\int_0^\frac{T}{2}\frac{1}{|t-t'|^\frac{1}{3}}\|F(u(t'))dt'\|_{L_x^\frac{6d}{3d+2}}\Big\|_{L_t^{4d+2}([T,2T]\times\Bbb R^d)}\\
\lesssim& T^{\frac{1}{4d+2}-\frac{1}{3}}\sum_{0<\tau<\frac{T}{4}}
\big\|F(u)\big\|_{L_t^1L_x^\frac{6d}{3d+2}([\tau,2\tau]\times\Bbb R^d)}\\
\lesssim& T^{\frac{1}{4d+2}-\frac{1}{3}}\sum_{0<\tau<\frac{T}{4}}\tau^\frac{1}{6}\big\|u\big\|_{L_t^\frac{18}{5}L_x^\frac{18d}{9d-10}([\tau,2\tau]\times\Bbb R^d)}^3\\
\lesssim&_u T^{\frac{1}{4d+2}-\frac{1}{6}},
\end{align*}
where the last inequality comes from (\ref{43}). Interpolating
between this estimate and (\ref{equ44}) with
$\theta=\frac{2d(d-1)}{4d^2-5d-2}$, we obtain that
\begin{equation*}\Big\|P_{BT^{-\frac{1}{2}}}e^{i(t-\frac{T}{2})\Delta}u(\frac{T}{2})\Big\|_{L_t^6L_x^\frac{6d}{3d-2}([T,2T]\times{\Bbb
R}^d)}\lesssim_u B^{-\frac{2(d-1)(d-2)}{3(4d^2-5d-2)}}.
\end{equation*}
Summing this over dyadic $B\geq A$ yields (\ref{equ42}) and
(\ref{equ41}).

Finally, we explain why (\ref{equ41}) implies
$\mathcal{S}(A)\lesssim_u A^{-\frac{1}{2d+6}}$. Choosing
$\eta=\frac{1}{2}$. Then there exists $A_0$ depending on $u$, so
that (\ref{equ41}) holds for all $A\geq A_0$. By (\ref{41}), we need
only bound $S(A)$ for $A\geq A_0$.

Choose $k\geq1$ so that $2^{-4k}A\leq A_0\leq 2^{-4(k-1)}A$. By
iterating (\ref{equ41}) $k$ times and using (\ref{41}),
\begin{align*}
S(A)\leq&\big[1+\eta2^\frac{4}{2d+6}+\cdots+(\eta2^\frac{4}{2d+6})\big]A^{-\frac{1}{2d+6}}+\eta^kS(2^{-4k}A)\\
\lesssim&_u A^{-\frac{1}{2d+6}}+\eta^k\lesssim_u
A^{-\frac{1}{2d+6}}.
\end{align*}
Note that the last inequality uses the way we choose $\eta$ and $k$.
\end{proof}
\begin{corollary}\label{co}
For any $A>0$, we have
\begin{equation*}
\mathcal{M}(A)+\mathcal{N}(A)+\mathcal{S}(A)\lesssim_u
A^{-\frac{1}{2d+6}}.
\end{equation*}
\end{corollary}
\begin{proof}
The bound on $\mathcal{S}(A)$ was derived in Proposition
\ref{pro51}. This together with Lemma \ref{NE} and (\ref{41}) yields
the bound on $\mathcal{N}(A)$.

We now turn to the bound on $\mathcal{M}(A)$. By Corollary
\ref{co1},
\begin{equation}\label{equ51}
\|P_{>AT^{-1/2}}u(T)\|_{L_x^2}\lesssim\sum_{k=0}^\infty\Big\|\int_{2^kT}^{2^{k+1}T}e^{i(T-t')\Delta}P_{>AT^{-1/2}}F(u(t'))dt'\Big\|_{L_x^2},
\end{equation}
where weak convergence has become strong convergence because of the
frequency projection and the fact that $N(t)=t^{-1/2}\rightarrow 0$
as $t\rightarrow\infty$. Combining (\ref{equ51}) with Strichartz
estimates, (\ref{m}) and (\ref{n}), we get
\begin{equation}\label{equ52}
\mathcal{M}(A)=\sup_{T>0}\|P_{>AT^{-1/2}}u(T)\|_{L_x^2}\lesssim\sum_{k=0}^\infty\mathcal{N}(2^{k/2}A).
\end{equation}
The desired bound on $\mathcal{M}(A)$ now follows from that on
$\mathcal{N}(A)$.
\end{proof}
{\it Proof of Theorem \ref{self-similar}.} Combining Lemma
\ref{NE} with Corollary \ref{co}, one has
\begin{equation*}
\mathcal{N}(A)\lesssim_u
A^{-\frac{1}{4d+12}}[\mathcal{S}(\frac{A}{8})+\mathcal{M}(\frac{A}{8})+\mathcal{N}(\frac{A}{8})].
\end{equation*}
Together with (\ref{equ51}) and (\ref{equ52}), this allows us to
deduce
$$\mathcal{S}(A)+\mathcal{M}(A)+\mathcal{N}(A)\lesssim_u A^{-\sigma}\quad\Rightarrow\quad
\mathcal{S}(A)+\mathcal{M}(A)+\mathcal{N}(A)\lesssim_u
A^{-\sigma-\frac{1}{4d+12}}$$ for any $\sigma>0$. Iterating this
statement shows that $u(t)\in H_x^s(\Bbb R^d)$ for all $s>0$.
\begin{corollary}[Absence of self-similar solutions] There are no
non-zero spherically symmetric solutions to (\ref{har}) that are
self-similar in the sense of Theorem \ref{3sc}.
\end{corollary}
\begin{proof}
By Theorem \ref{self-similar}, any such solution would obey $u(t)\in
H_x^1(\Bbb R^d)$ for all $t\in (0, \infty)$. Then there exists a
global solution with initial data $u(t_0)$ at any time $t_0\in (0,
\infty)$; recall that we assume $M(u)<M(Q)$ in the focusing case
(see \cite{MiXZ1}, \cite{MiXZ4}). On the other hand, self-similar
solutions blow up at time $t=0$. These two facts yield a
contradiction.
\end{proof}
\section{Additional regularities} This section is devoted to the proof
of Theorem \ref{tm51}. Before giving the proof, we record some basic
local estimates. From mass conservation we have
\begin{equation}\label{51}\|u\|_{L_t^\infty L_x^2}\lesssim_u 1,\end{equation}
while from Definition \ref{aps} and the fact that $N(t)$ is bounded,
we have
\begin{equation}\label{52}\lim_{N\rightarrow\infty}\|u_{\geq N}\|_{L_t^\infty L_x^2(\Bbb R\times{\Bbb R}^d)}=0.\end{equation}
From Proposition \ref{sb} and $N(t)\lesssim1$, we have
\begin{equation}\label{53}\|u\|_{L_t^6L_x^\frac{6d}{3d-2}(J\times{\Bbb R}^d)}\lesssim_u \langle|J|\rangle^\frac{1}{6}\end{equation}
for all intervals $J\subset\Bbb R$. By H\"{o}lder's inequality and
Hardy-Littlewood-Sobolev inequality, this implies
\begin{equation}\label{54}\|F(u)\|_{L_t^2L_x^\frac{2d}{d+2}(J\times{\Bbb R}^d)}\lesssim_u \langle|J|\rangle^{\frac{1}{2}}\end{equation}
and then by Strichartz estimates (Lemma \ref{stri}),
\begin{equation}\label{55}\|u\|_{L_t^qL_x^r(J\times{\Bbb R}^d)}\lesssim \langle|J|\rangle^\frac{1}{q},\end{equation}
for any admissible pair $(q, r)$. Similarly, the weighted Strichartz
estimates imply that
\begin{equation}\label{56}\||x|^\frac{2(d-1)}{q}u\|_{L_t^q L_x^\frac{2q}{q-4}(J\times\Bbb R^d)}\lesssim_u \langle|J|\rangle^\frac{1}{q}.\end{equation}
Now for any dyadic number $N$, define
\begin{equation}\label{58}\mathcal{M}(N):=\|u_{\geq N}\|_{L_t^\infty L_x^2(\Bbb R\times{\Bbb R}^d)},\end{equation}
then we see that $\mathcal{M}(N)\lesssim_u 1$ and
\begin{equation}\label{59}\lim_{N\rightarrow\infty}\mathcal{M}(N)=0.\end{equation}

To prove Theorem \ref{tm51}, it suffices to show that
$\mathcal{M}(N)\lesssim_{u,s} N^{-s}$ for any $s>0$ and all $N$
sufficiently large depending on $u$ and $s$. This will immediately
follow from iterating the following proposition with a suitably
choice of small $\eta$ (depending on $u$ and $s$):
\begin{proposition}[Regularity] \label{tm52}Let $u$ be as in Theorem \ref{tm51}
and let $\eta>0$ be a small number. Then
\begin{equation}\label{510}\mathcal{M}(N)\lesssim_u \eta\mathcal{M}\big(\frac{N}{64}\big)\end{equation}
whenever $N$ is sufficiently large depending on $u$ and $\eta$.
\end{proposition}

The rest of this section is devoted to proving Proposition
\ref{tm52}. Our task is to show that $$\|u_{\geq
N}(t_0)\|_{L_x^2{({\Bbb R}^d)}}\lesssim_u \eta
\mathcal{M}\Big(\frac{N}{64}\Big)$$ for all times $t_0$ and all $N$
sufficiently large (depending on $u$ and $\eta$). By time
translation symmetry, we may assume $t_0=0$. 
By Corollary 1.1, we have
\begin{align}\label{511}
u_{\geq N}(0)=&(P^++P^-)u_{\geq N}(0)\\
=&\lim_{T\rightarrow\infty}i\int_0^TP^+e^{-it\Delta}P_{\geq
N}F(u(t))dt-i\lim_{T\rightarrow\infty}\int_{-T}^0P^-e^{-it\Delta}P_{\geq
N}F(u(t))dt,\nonumber\end{align} where the limit is to be
interpreted as a weak limit in $L^2$. However, this representation
is not useful for $|x|$ small because the kernels of $P^\pm$ have a
logarithmic singularity at $x=0$. To deal with this, we will use a
different representation for $|x|\leq N^{-1}$, namely
\begin{equation}\label{512}u_{\geq N}(0)=\lim_{T\rightarrow\infty}i\int_0^Te^{-it\Delta}P_{\geq N}F(u(t))dt,\end{equation}
also as a weak limit. To deal with the poor nature of these limits,
we use the fact that
\begin{equation}\label{513}f_T\rightarrow f\quad \text{weakly along a subsequence}\Rightarrow
\|f\|\leq\limsup_{T\rightarrow\infty}\|f_T\|,\end{equation} or
equivalently, that the unit ball is weakly closed.

Despite different representations will be used depending on the size
of $|x|$, some estimates can be dealt with in a uniform manner. The
first such example is a bound on integrals over short times.
\begin{lemma}[Local estimate]\label{tm53} For any $\eta>0$, there
exists $\delta=\delta(u,\eta)>0$ such that
$$\Big\|\int_0^\delta e^{-it\Delta}P_{\geq N}F(u(t))dt\Big\|_{L_x^2}\lesssim_u \eta\mathcal{M}\big(\frac{N}{8}\big),$$
provided that $N$ is sufficiently large depending on $u$ and $\eta$.
An analogous estimate holds for integration over $[-\delta,0]$ and
after pre-multiplication by $P^\pm$ {\rm (}they are bounded
operators on $L_x^2${\rm )}.
\end{lemma}
\begin{proof} By Strichartz estimates, it suffices to prove
$$\|P_{\geq N}F(u)\|_{L_t^2 L_x^\frac{2d}{d+2}(J\times{\Bbb R}^d)}\lesssim_u \eta\mathcal{M}\big(\frac{N}{8}\big)$$
for any interval $J$ with length $|J|\leq\delta$ and all
sufficiently large $N$ depending on $u$ and $\eta$.

 From (\ref{52}),
there exists $N_0=N_0(u,\eta)$ such that
\begin{equation}\label{514}\|u_{\geq N_0}\|_{L_t^\infty L_x^2(\Bbb R\times{\Bbb R}^d)}\leq \eta^2.\end{equation}
Let $N>8N_0$. We decompose
$$u=u_{\geq\frac{N}{8}}+u_{N_0\leq\cdot\leq\frac{N}{8}}+u_{<N_0}.$$
Any term in the resulting expansion of $P_{\geq N}F(u)$ that doesn't
contain the factor of $u_{\geq\frac{N}{8}}$ vanishes.

At first, we consider the terms with two factors of the form
$u_{<N_0}$. Using H\"{o}lder's inequality, Hardy-Littlewood-Sobolev
inequality, (\ref{51}) and Lemma \ref{Bern},
\begin{align*}\big\|\big(V*(u_{>\frac{N}{8}}u_{<N_0})\big)u_{<N_0}\big\|_{L_t^2L_x^\frac{2d}{d+2}(J\times{\Bbb R}^d)}
\lesssim&\|V*(u_{>\frac{N}{8}}u_{<N_0})\|_{L_t^4L_x^\frac{2d}{3}(J\times\Bbb
R^d)}\|u_{<N_0}\|_{L_t^4L_x^\frac{2d}{d-1}(J\times\Bbb R^d)}\\
\lesssim&\|u_{>\frac{N}{8}}u_{<N_0}\|_{L_t^4L_x^\frac{2d}{2d-1}(J\times\Bbb
R^d)}\|u_{<N_0}\|_{L_t^4L_x^\frac{2d}{d-1}(J\times\Bbb R^d)}\\
\lesssim&\|u_{>\frac{N}{8}}\|_{L_t^\infty L_x^2}\|u_{<N_0}\|_{L_t^4L_x^\frac{2d}{d-1}(J\times\Bbb
R^d)}^2\\
\lesssim&|J|^\frac{1}{2}N_0\mathcal{M}(\frac{N}{8})
\end{align*}
and
\begin{align*}
\big\|\big(V*u_{<N_0}^2\big)u_{>\frac{N}{8}}\big\|_{L_t^2L_x^\frac{2d}{d+2}(J\times{\Bbb
R}^d)}
\lesssim&\big\|V*u_{<N_0}^2\big\|_{L_t^2L_x^d(J\times{\Bbb R}^d)}\big\|u_{>\frac{N}{8}}\big\|_{L_t^\infty L_x^2(\Bbb R\times{\Bbb R}^d)}\\
\lesssim&\big\|u_{<N_0}\big\|_{L_t^4L_x^\frac{2d}{d-1}(J\times{\Bbb R}^d)}^2\mathcal{M}(\frac{N}{8})\\
\lesssim&|J|^\frac{1}{2}N_0\mathcal{M}(\frac{N}{8}).
\end{align*}
Choosing $\delta$ sufficiently small depending on $\eta$ and $N_0$,
we see that they are acceptable.

It remains only to consider those components of $P_{\geq N}F(u)$
which involve $u_{\geq\frac{N}{8}}$ and at least one of the other
terms is not $u_{<N_0}$.
\begin{align*}
\big\|V*(u_{>\frac{N}{8}}u_{\geq
N_0})u\big\|_{L_t^2L_x^\frac{2d}{d+2}(J\times{\Bbb R}^d)}
\lesssim&\big\|V*(u_{>\frac{N}{8}}u_{\geq N_0})\big\|_{L_t^4
L_x^\frac{2d}{3}(J\times{\Bbb R}^d)}\|u\|_{L_t^4L_x^\frac{2d}{d-1}(J\times{\Bbb R}^d)}\\
\lesssim&\big\|u_{>\frac{N}{8}}\big\|_{L_t^\infty L_x^2(J\times{\Bbb
R}^d)}\big\|u_{\geq N_0}\big\|_{L_t^4
L_x^\frac{2d}{d-1}(J\times{\Bbb
R}^d)}\langle|J|\rangle^{\frac{1}{4}}.
\end{align*}
By (\ref{55}), we get
$$\|u_{>N_0}\|_{L_t^2L_x^\frac{2d}{d-2}(J\times\Bbb R^d)}\lesssim\langle|J|\rangle^\frac{1}{2}.$$
Therefore, by interpolation between this and (\ref{514}), we have
\begin{equation*}
\|u_{>N_0}\|_{L_t^4L_x^\frac{2d}{d-1}(J\times\Bbb
R^d)}\lesssim\eta\langle|J|\rangle^\frac{1}{4}.
\end{equation*}
Thus, we obtain that
\begin{equation*}
\big\|V*(u_{>\frac{N}{8}}u_{\geq
N_0})u\big\|_{L_t^2L_x^\frac{2d}{d+2}(J\times{\Bbb
R}^d)}\lesssim\eta\mathcal{M}\big(\frac{N}{8}\big)\langle|J|\rangle^\frac{1}{2}.
\end{equation*}
Similarly, we can estimate
\begin{align*}\big\|V*\big(u_{>\frac{N}{8}}u\big)u_{\geq N_0}\big\|_{L_t^2L_x^\frac{2d}{d+2}(J\times{\Bbb R}^d)}
\lesssim&\big\|V*\big(u_{>\frac{N}{8}}u\big)\big\|_{L_t^2L_x^d(J\times{\Bbb
R}^d)}\big\|u_{\geq
N_0}\big\|_{L_t^\infty L_x^2(J\times{\Bbb R}^d)}\\
\lesssim&\big\|u_{>\frac{N}{8}}u\big\|_{L_t^2L_x^\frac{d}{d-1}(J\times{\Bbb
R}^d)}\big\|u_{\geq
N_0}\big\|_{L_t^\infty L_x^2(J\times{\Bbb R}^d)}\\
\lesssim&\big\|u_{>\frac{N}{8}}\big\|_{L_t^\infty L_x^2(J\times{\Bbb
R}^d)}\big\|u_{\geq N_0}\big\|_{L_t^\infty L_x^2
(J\times{\Bbb R}^d)}\|u\|_{L_t^2L_x^\frac{2d}{d-2}(J\times{\Bbb R}^d)}\\
\lesssim&\eta^2\mathcal{M}\big(\frac{N}{8}\big)\langle|J|\rangle^{\frac{1}{2}},
\end{align*}
where the last inequality comes from (\ref{514}). Another term
$\big\|V*\big(u_{\geq
N_0}u\big)u_{>\frac{N}{8}}\big\|_{L_t^2L_x^\frac{2d}{d+2}(J\times{\Bbb
R}^d)}$ can be similarly estimated.
\end{proof}

We now turn our attention to $|t|\geq\delta$. In this case we make
the decomposition $$P_{\geq N}=\sum_{M\geq N}P_M\tilde{P}_M,$$ where
$\tilde{P}_M:=P_{M/2}+P_M+P_{2M}$. In this way, (\ref{512}) becomes
\begin{align}\label{31}u_{\geq N}(0,x)=&\ i\int_0^\delta e^{-it\Delta}P_{\geq
N}F(u(t))dt\nonumber\\
&\quad+\lim_{T\rightarrow\infty}\sum_{M\geq
N}i\int_\delta^T\int_{{\Bbb
R}^d}[P_Me^{-it\Delta}](x,y)[\tilde{P}_MF(u(t))](y)dydt,
\end{align}
which we will use when $|x|\leq N^{-1}$. The analogous reformulation
of (\ref{511}), namely
\begin{align}\label{32}
u_{\geq N}(0,x)=&\ i\int_0^\delta P^+e^{-it\Delta}P_{\geq
N}F(u(t))dt-i\int_{-\delta}^0 P^-e^{-it\Delta}P_{\geq
N}F(u(t))dt\nonumber\\
&+\lim_{T\rightarrow\infty}\sum_{M\geq N}i\int_\delta^T\int_{{\Bbb R}^d}[P_M^+e^{-it\Delta}](x,y)[\tilde{P}_MF(u(t))](y)dydt\\
&-\lim_{T\rightarrow\infty}\sum_{M\geq
N}i\int^{-\delta}_{-T}\int_{{\Bbb
R}^d}[P_M^-e^{-it\Delta}](x,y)[\tilde{P}_MF(u(t))](y)dydt,\nonumber
\end{align}
will be used when $|x|>N^{-1}$.

The next lemma bounds the integrals over the significant region
$|y|\gtrsim M|t|$. Let $\chi_k$ denote the characteristic function
of the set $$\{(t,y): 2^k\delta\leq|t|\leq 2^{k+1}\delta, |y|\gtrsim
M|t|\}.$$
\begin{lemma}[Main contribution]\label{tm54}Let $\eta>0$ be a small
number and let $\delta$ be as in Lemma \ref{tm53}. Then
\begin{equation}\label{515}\sum_{M\geq
N}\sum_{k=0}^\infty\Big\|\int_{2^k\delta}^{2^{k+1}\delta}
\int_{{\Bbb
R}^d}[P_Me^{-it\Delta}](x,y)\chi_k(t,y)[\tilde{P}_MF(u(t))](y)
dydt\Big\|_{L_x^2}\lesssim_u
\eta\mathcal{M}(\frac{N}{64})\end{equation} for all $N$ sufficiently
large depending on $u$ and $\eta$. An analogous estimate holds with
$P_M$ replaced by $P_M^+$ and $P_M^-$; moreover, the time integrals
may be taken over $[-2^{k+1}\delta, -2^k\delta]$.
\end{lemma}
\begin{proof} By Strichartz estimates, we get
\begin{align*}
\Big\|\int_{2^k\delta}^{2^{k+1}\delta} \int_{{\Bbb
R}^d}[P_Me^{-it\Delta}](x,y)&\chi_k(t,y)[\tilde{P}_MF(u(t))](y)
dydt\Big\|_{L_x^2}\\
\lesssim&\
\big\|\chi_k(V*|u|^2)\tilde{P}_{>\frac{M}{8}}u\big\|_{L_t^1L_y^2}+\big\|\chi_ku\tilde{P}_{>\frac{M}{8}}(V*|u|^2)\big\|_{L_t^2L_y^\frac{2d}{d+2}}.
\end{align*}
We first consider \begin{align*}
\Big\|\chi_k(V*|u|^2)\tilde{P}_{>\frac{M}{8}}u\Big\|_{L_t^1L_y^2}\
\lesssim\
&\Big\|\chi_k(V*|u|^2)\Big\|_{L_t^1L_y^\infty}\Big\|P_{>\frac{M}{64}}u\Big\|_{L_t^\infty
L_y^2}\\
\lesssim\ &\
\Bigg\|\chi_k(y)\int_{|x-y|\geq\frac{|y|}{2}}\frac{|u(x)|^2}{|y-x|^2}dx\Bigg\|_{L_t^1L_y^\infty}\mathcal{M}(\frac{N}{64})
\\
&+\Bigg\|\chi_k(y)\int_{|x-y|<\frac{|y|}{2}}\frac{|u(x)|^2}{|y-x|^2}dx\Bigg\|_{L_t^1L_y^\infty}\mathcal{M}(\frac{N}{64}).
\end{align*}
On one hand, \begin{align*}
\Bigg\|\chi_k(y)\int_{|x-y|\geq\frac{|y|}{2}}\frac{|u(x)|^2}{|y-x|^2}dx\Bigg\|_{L_t^1L_y^\infty}
\lesssim&\ \big\|\chi_k(y)|y|^{-2}\big\|_{L_t^1L_y^\infty}\|u\|_{L_t^\infty L_y^2}\\
\lesssim&_u\ (M2^k\delta)^{-2}(2^k\delta).
\end{align*}
On the other hand, by weighted Strichartz estimates, H\"{o}lder's
inequality and (\ref{56}),
taking $p=\frac{d}{2}-\frac{1}{4}$, $q=2d-\frac{3}{2}$ and
$\theta=\frac{q}{2p}=\frac{4d-3}{2d-1}\leq2$, we have
\begin{align*}
\Bigg\|\chi_k(y)\int_{|x-y|<\frac{|y|}{2}}&\ \frac{|u(x)|^2}{|y-x|^2}dx\Bigg\|_{L_t^1L_y^\infty}\\
\lesssim &\
\Bigg\|\chi_k(y)|y|^{-\frac{2(d-1)}{q}\theta}\int_{|x-y|\leq\frac{|y|}{2}}
\frac{|y|^{\frac{2(d-1)}{q}\theta}|u(x)|^2}{|x-y|^2}dx\Bigg\|_{L_t^1L_y^\infty}\\
\lesssim&\
\Bigg\|\chi_k(y)|y|^{-\frac{2(d-1)}{q}\theta}\int_{|x-y|\leq\frac{|y|}{2}}
\frac{|x|^{\frac{2(d-1)}{q}\theta}|u(x)|^2}{|x-y|^2}dx\Bigg\|_{L_t^1L_y^\infty}
\\
\lesssim&\
\Bigg\|\chi_k(y)|y|^{-\frac{2(d-1)}{q}\theta}\Big\|1_{|\cdot|\leq\frac{|y|}{2}}\frac{1}{|\cdot|^2}
\Big\|_{L_x^p}\Big\||x|^{\frac{2(d-1)}{q}}|u|\Big\|_{L_x^\frac{2q}{q-4}}^\theta\|u\|_{L_x^2}^{2-\theta}\Bigg\|_{L_t^1L_y^\infty}\\
\lesssim&\
\Big\|\chi_k(y)|y|^{-\frac{2(d-1)}{q}\theta}|y|^\frac{d-2p}{p}\Big\|_{L_t^\frac{2p}{2p-1}L_y^\infty}
\big\||x|^{\frac{2(d-1)}{q}}|u|\big\|_{L_t^qL_x^\frac{2q}{q-4}}^\theta\|u\|_{L_t^\infty
L_x^2}^{2-\theta}\\
\lesssim_u &\
\Big(M2^k\delta\Big)^{-\frac{d-1}{p}+\frac{d-2p}{p}}(2^k\delta)^\frac{2p-1}{2p}\langle2^k\delta\rangle^\frac{\theta}{q}
\lesssim_u M^{-\frac{2p-1}{p}}(2^k\delta)^\frac{1-p}{p}.
\end{align*} Therefore, we have
\begin{equation}\Big\|\chi_ku_{>\frac{M}{8}}(V*|u|^2)\Big\|_{L_t^1L_y^2}\lesssim M^{-\frac{2(2d-3)}{2d-1}}
2^{-\frac{2d-5}{2d-1}k}\delta^{-\frac{2d-5}{2d-1}}\mathcal{M}(\frac{N}{64}).\end{equation}
At last, by means of Bernstein, weighted Strichartz estimates  and
(\ref{56}), we have
\begin{align*}
\Big\|\chi_ku\tilde{P}_{>\frac{M}{8}}(V*&|u|^2)\Big\|_{L_t^2L_y^\frac{2d}{d+2}}\\
\lesssim&\
\big\|\chi_ku\big\|_{L_t^2L_y^\infty}\Big\|\tilde{P}_{>\frac{M}{8}}(V*|u|^2)\Big\|_{L_t^\infty
L_y^\frac{2d}{d+2}}\\
\lesssim&\
M^{-\frac{d-2}{2}}\big\|\chi_k|y|^{-\frac{d-1}{2}}\big\|_{L_t^4L_y^\infty}
\big\||y|^{\frac{d-1}{2}}u\big\|_{L_t^4L_y^\infty}\big\|\tilde{P}_{>\frac{M}{8}}(|u|^2)\big\|_{L_t^\infty
L_y^1}\\
\lesssim&\
M^{-\frac{d-2}{2}}(M2^k\delta)^{-\frac{d-1}{2}}(2^k\delta)^\frac{1}{4}
\langle2^k\delta\rangle^{\frac{1}{4}}\|P_{>\frac{M}{64}}u\|_{L_t^\infty L_y^2}\|u\|_{L_t^\infty L_y^2}\\
\lesssim&_uM^{-\frac{d-2}{2}}(M2^k\delta)^{-\frac{d-1}{2}}(2^k\delta)^\frac{1}{4}
\langle2^k\delta\rangle^{\frac{1}{4}}\mathcal{M}\big(\frac{N}{64}\big)\\
=&\
M^{-\frac{2d-3}{2}}2^{-\frac{d-2}{2}k}\delta^{-\frac{d-2}{2}}\mathcal{M}\big(\frac{N}{64}\big).
\end{align*}
Thus the LHS of (\ref{515}) can be bounded by
$$(N^{-\frac{2(2d-3)}{2d-1}}\delta^{-\frac{2d-5}{2d-1}}+N^{-\frac{2d-3}{2}}\delta^{-\frac{d-2}{2}})\mathcal{M}(\frac{N}{64}).$$
This is acceptable as long as we choose $N$ sufficiently large
depending on $\delta$ and $\eta$.\end{proof}

Now we turn to the integration over the region of $(t,y)$ where
$|y|\ll M|t|$. In \cite{KVZ}, the bounds of the kernels of the
propagators have been shown to be
\begin{equation}\label{516}
|P_Me^{-it\Delta}(x,y)|+|P_M^\pm e^{-it\Delta}(x,y)|\lesssim
\frac{1}{(M^2|t|)^{50d}}K_M(x,y)
\end{equation}
where
\begin{equation*}
K_M(x,y):=\frac{M^d}{\langle M(x-y)\rangle^{50d}}+\frac{M^d}{\langle
Mx\rangle^{\frac{d-1}{2}}\langle My\rangle^{\frac{d-1}{2}}\langle
M|x|-M|y|\rangle^{50d}}.
\end{equation*}
Furthermore, by Schur's test, it is the kernel of a bounded operator
on $L_x^2(\Bbb R^d)$.

 Let $\tilde{\chi}_k$ be the
characteristic function of the set
$$\{(t,y):2^k\delta\leq|t|\leq2^{k+1}\delta, |y|\ll M|t|\}.$$
\begin{lemma}[The tail]\label{tm55} Let $\eta>0$ be a small number and let
$\delta$ be as in Lemma \ref{tm53}. Then
$$\sum_{M\geq N}\sum_{k=0}^\infty\Big\|\int_{2^k\delta}^{2^{k+1}\delta}\int_{{\Bbb R}^d}
\frac{K_M(x,y)}{(M^2|t|)^{50d}}\tilde{\chi}_k(t,y)[\tilde{P}_MF(u(t))](y)dydt\Big\|_{L_x^2}\lesssim_u\eta\mathcal{M}\Big(\frac{N}{8}\Big)$$
for all $N$ sufficiently large depending on $u$ and $\eta$.
\end{lemma}
\begin{proof}By Strichartz estimates, Hardy-Littlewood-Sobolev
inequality, H\"older inequality and (\ref{55}), we have
\begin{align*}
\Big\|\int_{\Bbb R}\int_{{\Bbb R}^d}
\frac{K_M(x,y)}{(M^2|t|)^{50d}}\tilde{\chi}_k(t,y)&[\tilde{P}_MF(u(t))](y)dydt\Big\|_{L_x^2}\\
\lesssim&
(M^22^k\delta)^{-50d}\big\|\tilde{\chi}_k\tilde{P}_MF(u)\big\|_{L_t^1L_x^2}\\
\lesssim&(M^22^k\delta)^{-50d}\|u_{\geq \frac{M}{8}}\|_{L_t^\infty
L_x^2}\|u\|_{L_t^2L_x^\frac{2d}{d-2}([2^k\delta,2^{k+1}\delta]\times\Bbb
R^d)}^2\\
\lesssim&(M^22^k\delta)^{-50d}\mathcal{M}\Big(\frac{N}{8}\Big)\langle2^k\delta\rangle.
\end{align*}
Summing over $k\geq0$ and $M\geq N$, we get
$$\sum_{M\geq N}\sum_{k=0}^\infty\Big\|\int_{\Bbb R}\int_{{\Bbb R}^d}
\frac{K_M(x,y)}{(M^2|t|)^{50d}}\tilde{\chi}_k(t,y)[\tilde{P}_MF(u(t))](y)dydt\Big\|_{L_x^2}\lesssim_u(N^2\delta)^{-49d}\mathcal{M}\Big(\frac{N}{8}\Big).$$
The claim follows by choosing $N$ sufficiently large depending on
$\delta$ and $\eta$.\end{proof}

{\it Proof of Proposition \ref{tm53}.} Naturally, we may bound
$\|u_{\geq N}\|_{L^2}$ by separately bounding the $L^2$ norm on the
ball $\{|x|\leq N^{-1}\}$ and on its complement. On the ball we use
(\ref{31}) while outside the ball we use (\ref{32}). Invoking
(\ref{513}) and the triangle inequality, we reduce the proof to
bounding certain integrals. The integrals over short times were
estimated in Lemma \ref{tm53}. For $|t|\geq\delta$, we further
partition the region of integration into two main pieces. The first
piece, where $|y|\gtrsim M|t|$, was dealt with in Lemma \ref{tm54}.
The remaining piece, $|y|\ll M|t|$, can be estimated by combining
(\ref{516}) with Lemma \ref{tm55}.
\section{Double high-to-low frequency cascade}
In this section, we use the additional regularity provided by
Theorem \ref{tm51} to preclude double high-to-low frequency cascade
solutions. We need the following lemma:

\begin{lemma}[Gagliardo-Nirenberg inequality of convolution
type, \cite{MiXZ4}]\label{LV}Let $V(x)=|x|^{-2}$ and
\begin{equation}\big\|u\big\|_{L^V}:=\Big(\iint_{\Bbb R^d\times{\Bbb
R}^d}
|u(x)|^{2}V(x-y)|u(y)|^{2}dxdy\Big)^{\frac{1}{4}},\end{equation}
 then \begin{equation}\label{gn}\|u\|_{L^V}^4\leq\frac{2}{\|Q\|_{L^2}^2}\|u\|_{L^2}^2\|\nabla
u\|_{L^2}^2.\end{equation}\end{lemma}
\begin{theorem}[Absence of double cascades]\label{tm61} There are no non-zero
global spherically symmetric solutions to (\ref{har}) that are
double high-to-low frequency cascades in the sense of Theorem
\ref{3sc}.
\end{theorem}
\begin{proof} Suppose there exists such a solution $u$. By Theorem
\ref{tm51}, $u$ lies in $C_t^0H_x^1(\Bbb R\times\Bbb R^d)$. Hence
the energy
$$E(u(t)):=\frac{1}{2}\int_{\Bbb R^d}|\nabla
u(t,x)|^2dx+\frac{\mu}{4}\iint_{\Bbb R^d\times\Bbb
R^d}\frac{|u(x)|^2|u(y)|^2}{|x-y|^2}dxdy$$ is finite and conserved.
As we have $M(u)<M(Q)$ in the focusing case, Lemma \ref{LV} gives
\begin{equation}\label{61}
\|\nabla u(t)\|_{L_x^2(\Bbb R^d)}^2\sim_u E(u)\sim_u1
\end{equation}
for all $t\in\Bbb R$. Since
$$\liminf_{t\rightarrow-\infty}N(t)=\liminf_{t\rightarrow+\infty}N(t)=0,$$
there are two time sequences along which $N(t)\rightarrow0$.

Let $\eta>0$ be arbitrary. By Definition \ref{aps}, we can find
$C=C(\eta,u)>0$ such that
$$\int_{|\xi|\geq CN(t)}|\hat{u}(t,\xi)|^2d\xi\leq \eta^2$$
for all $t$. Meanwhile, by Theorem \ref{tm51}, $u\in C_t^0H_x^s(\Bbb
R\times\Bbb R^2)$ for some $s>1$. Thus, $$\int_{|\xi|\geq
CN(t)}|\xi|^{2s}|\hat{u}(t,\xi)|^2d\xi\lesssim_u 1$$ for all $t$ and
some $s>1$. By H\"older inequality, we obtain
$$\int_{|\xi|\geq CN(t)}|\xi|^2|\hat{u}(t,\xi)|^2d\xi\lesssim_u \eta^{2(s-1)/s}.$$
On the other hand, from mass conservation and Plancherel's theorem,
we obtain
$$\|\nabla u(t)\|_{L_x^2(\Bbb R^d)}\lesssim_u \eta^{(s-1)/s}+CN(t)$$
for all $t$. As $\eta>0$ is arbitrary and there exists a sequence of
times $t_n\rightarrow\pm\infty$ such that $N(t_n)\rightarrow 0$, we
conclude that $\|\nabla u(t_n)\|_{L_x^2(\Bbb R^d)}\rightarrow 0$, as
$n\rightarrow\infty$. This contradicts with (\ref{61}). \end{proof}
 \section{Death of solitons}
 Let $$M_a(t):=2\Im\int_{\Bbb
 R^d}\bar{u}(t,x)\vec{a}(x)\cdot\nabla u(t,x)dx,$$
 then we have (see \cite{MiXZ2} for similar calculation)
\begin{align*}
\partial_tM_a(t)=&-\int_{\Bbb R^d}\Delta(\partial_ja_j)|u(t,x)|^2dx+4\Re\int_{\Bbb R^d}\partial_ka_{j}u_ju_kdx\\
&\quad\quad-\mu\iint_{\Bbb R^d\times\Bbb
R^d}(\vec{a}(x)-\vec{a}(y))\cdot\nabla
V(x-y)|u(t,x)|^2|u(t,y)|^2dxdy.
\end{align*}
\begin{lemma}[Localized virial identity]\label{tm71}
 Let $\vec{a}(x)=x\psi\big(\frac{|x|}{R}\big)$, where $\psi$ is a smooth
 function and $\psi(r)=1$ when $r\leq1$; $\psi=0$ when $r\geq2$. Then we get
\begin{align}
\partial_tM_a(t)=&8E(u(t))\nonumber\\
&-\int_{\Bbb
R^d}\Big[\frac{d^2-1}{R|x|}\psi'\big(\frac{|x|}{R}\big)+\frac{2d+1}{R^2}\psi''\big(\frac{|x|}{R}\big)+
\frac{|x|}{R^3}\psi'''\big(\frac{|x|}{R}\big)\Big]|u(t,x)|^2dx\label{71}\\
&+4\int_{\Bbb
R^d}\Big[\psi\big(\frac{|x|}{R}\big)-1+\frac{|x|}{R}\psi'\big(\frac{|x|}{R}\big)\Big]|\nabla
u(t,x)|^2dx\label{72}\\
&+2\mu\iint_{\Bbb R^d\times\Bbb
R^d}\Big[x\psi\big(\frac{|x|}{R}\big)-y\psi\big(\frac{|y|}{R}\big)-(x-y)\Big]\cdot\frac{x-y}{|x-y|^4}|u(t,x)|^2|u(t,y)|^2dxdy.\label{73}
\end{align}
\end{lemma}
 \begin{proposition}
 There are no non-zero global spherically symmetric solutions to
 (\ref{har}) that are soliton-like in the sense of Theorem \ref{3sc}.
 \end{proposition}
 \begin{proof} Assume to the contrary that there is such a solution $u$.
 Then by Theorem \ref{tm51}, $u\in C_t^0H_x^s$ for some $s>1$. In particular,
 \begin{equation}\label{74}
 |M_a(t)|\lesssim_u R.
 \end{equation} Note that in the focusing case, $M(u)<M(Q)$. As a
 consequence, Lemma \ref{LV} gives \begin{equation}\label{91}E(u)\gtrsim_u\int_{\Bbb R^d}|\nabla
 u(t,x)|^2>0.\end{equation}
 We will show that (\ref{71})-(\ref{73}) constitute only a small
 fraction of $E(u)$. Combining this fact with Lemma \ref{tm71}, we
 conclude $\partial_tM_a(t)\gtrsim E(u)>0$, which contradicts with (\ref{74}).
 As in \cite{KVZ}, (\ref{71}) and (\ref{72}) can be bounded by $R^{-2}$ and
 $\eta^\frac{s-1}{s}+\eta$, respectively, where $\eta>0$ is a small number to be chosen later.
 The rest of this section is devoted to estimating (\ref{73}).

By Definition \ref{aps} and the fact that $N(t)=1$ for all $t\in\Bbb
R$, if $R$ is sufficiently large depending on $u$ and $\eta$, then
\begin{equation}\label{92}\int_{|x|\geq\frac{R}{4}}|u(t,x)|^2dx\leq\eta\end{equation} for all
$t\in\Bbb R$. Let $\chi$ denote a smooth cutoff to the region
$|x|\geq\frac{R}{2}$, chosen so that $\nabla \chi$ is bounded by
$R^{-1}$ and supported where $|x|\sim R$. By Lemma \ref{LV},
(\ref{91}) and (\ref{92}), we have
\begin{align*} (\ref{73})
\leq& C\iint_{|x|\geq R\atop
|y|\geq R}\Big[x(\psi(\frac{|x|}{R})-1)-y(\psi(\frac{|y|}{R})-1)\Big]\cdot\frac{x-y}{|x-y|^4}|u(x)|^2|u(y)|^2dxdy\\
&\quad+C\iint_{|x|\geq R\atop |y|\leq
R}|x(\psi(\frac{|x|}{R})-1)|\frac{|u(x)|^2|u(y)|^2}{|x-y|^3}dxdy\\
&\quad+C\iint_{|y|\geq R\atop
|x|\leq R}|y(\psi(\frac{|y|}{R})-1)|\frac{|u(x)|^2|u(y)|^2}{|x-y|^3}dxdy\\
:=&I+II+III. \end{align*} On one hand,
\begin{align*}I\leq&C\iint_{|x|\geq R\atop
|y|\geq
R}\frac{|\chi(x)u(x)|^2|\chi(y)u(y)|^2}{|x-y|^2}dxdy\\
\lesssim&\|\chi u\|_{L_x^2}^2\|\nabla (\chi u)\|_{L_x^2}^2
\lesssim_u \eta.
\end{align*}
On the other hand,
\begin{align*}
II\leq& C\iint_{R<|x|\leq 2R\atop |y|\leq
R}\Big|x\big(\psi(\frac{|x|}{R})-1\big)\Big|\frac{|\chi(x)u(x)|^2|u(y)|^2}{|x-y|^3}dxdy\\
&\quad+\iint_{|x|>2R \atop |y|\leq
R}|x|\frac{|\chi(x)u(x)|^2|u(y)|^2}{|x-y|^3}dxdy\\
\leq& C\iint_{\Bbb R^d\times\Bbb
R^d}\frac{|\chi(x)u(x)|^2|u(y)|^2}{|x-y|^3}dxdy+\iint_{\Bbb
R^d\times\Bbb R^d}\frac{|\chi(x)u(x)|^2|u(y)|^2}{|x-y|^2}dxdy\\
\lesssim&\|\chi u\|_{L_x^2}\|\nabla (\chi u)\|_{L_x^2}\|\nabla
u\|_{L_x^2}^2+\|\chi u\|_{L_x^2}\|\nabla (\chi
u)\|_{L_x^2}\|u\|_{L_x^2}\|\nabla u\|_{L_x^2}\\
\lesssim&_u \eta^{1/2}.
\end{align*}
$III$ can be estimated similarly.  Choosing $\eta$ sufficiently
small depending on $u$ and $R$
 sufficiently large depending on $u$ and $\eta$, we obtain
 $$(\ref{71})+(\ref{72})+(\ref{73})\leq\frac{1}{100}E(u).$$
 This completes the proof.
\end{proof}

\vskip0.5cm
 \textbf{Acknowledgements:} The authors thank the
referees and the associated editor for their invaluable comments
and suggestions which helped improve the paper greatly. C. Miao
and G.Xu  were partly supported by the NSF of China (No.10725102,
No.10726053), and L.Zhao was supported by China postdoctoral
science foundation project.

  \begin{center}

\end{center}
\end{document}